\documentclass[11pt,a4paper]{article}
\pdfoutput=1 
\usepackage[margin=0.771in]{geometry}
\usepackage{amsmath, amssymb, amsthm, enumerate, bbm ,graphicx,color,caption,upgreek, float, tikz, subcaption,booktabs,longtable, appendix,graphics, pdfpages,rotating,mathtools, array, bm, blkarray,setspace,textcomp,pgfplots,ytableau, tabularx}
\usepackage{thm-restate}
\usepackage{enumitem}
\usepackage{siunitx}
\DeclarePairedDelimiter\floor{\lfloor}{\rfloor}
\usepackage[pdftex]{hyperref} 
\usepackage[capitalize]{cleveref}

\usetikzlibrary{arrows.meta}
\usetikzlibrary{backgrounds}
\usepgfplotslibrary{patchplots}
\usepgfplotslibrary{fillbetween}
\pgfplotsset{%
layers/standard/.define layer set={%
    background,axis background,axis grid,axis ticks,axis lines,axis tick labels,pre main,main,axis descriptions,axis foreground%
}{grid style= {/pgfplots/on layer=axis grid},%
    tick style= {/pgfplots/on layer=axis ticks},%
    axis line style= {/pgfplots/on layer=axis lines},%
    label style= {/pgfplots/on layer=axis descriptions},%
    legend style= {/pgfplots/on layer=axis descriptions},%
    title style= {/pgfplots/on layer=axis descriptions},%
    colorbar style= {/pgfplots/on layer=axis descriptions},%
    ticklabel style= {/pgfplots/on layer=axis tick labels},%
    axis background@ style={/pgfplots/on layer=axis background},%
    3d box foreground style={/pgfplots/on layer=axis foreground},%
    },
}

\definecolor{blauw}{RGB}{61,158,255}
\definecolor{donkerblauw}{RGB}{0,0,255}
\definecolor{donkergroen}{RGB}{46,148,0}
\definecolor{donkerrood}{RGB}{204,0,0}
\definecolor{donkergroen2}{RGB}{0,95,0}

\newcommand{\revise}[1]{#1}
\newcommand{\revisefinal}[1]{#1}
\makeatletter 
\newcommand\mynobreakpar{\par\nobreak\@afterheading} 
\makeatother

\makeatletter
\let\@fnsymbol\@arabic
\makeatother

\let\OLDthebibliography\thebibliography
\renewcommand\thebibliography[1]{
  \OLDthebibliography{#1}
  \setlength{\parskip}{0pt}
  \setlength{\itemsep}{0pt plus 0.3ex}
}

\newcommand{\N}{\mathbb{N}}
\newcommand{\Z}{\mathbb{Z}}
\newcommand{\C}{\mathbb{C}}
\newcommand{\R}{\mathbb{R}}

\usepackage[english]{babel}

\newtheorem{theorem}{Theorem}[section]
\newtheorem{lemma}[theorem]{Lemma}

\newtheorem{proposition}[theorem]{Proposition}
\newtheorem{corollary}[theorem]{Corollary}

\theoremstyle{definition}

\newtheorem{definition}{Definition}[section]

\newtheorem*{examp*}{Example}

\newtheorem{remark}{Remark}[section]

\DeclareMathOperator*{\cross}{cr}
\DeclareMathOperator*{\ind}{Ind}

\theoremstyle{plain}
\newfloat{Algorithm}{!hbt}{alg}
\newcommand{\T}{^{\sf T}}

\newcounter{thm}[section]

\pgfplotsset{compat=1.17}

\title{New lower bounds on crossing numbers of $K_{m,n}$ \\ from semidefinite programming}% \date{23-02-2021}
\author{Daniel Brosch\thanks{University of Klagenfurt and Tilburg University. \href{mailto: daniel.brosch@aau.at}{\texttt{daniel.brosch@aau.at}}.} \  \& Sven Polak\thanks{Tilburg University and Centrum Wiskunde \& Informatica, Amsterdam. \href{mailto: s.c.polak@tilburguniversity.edu}{\texttt{s.c.polak@tilburguniversity.edu}}.}}
\begin{document}
\maketitle
\setcounter{footnote}{1}

\noindent \textbf{Abstract.} 
In this paper, we use semidefinite programming and representation theory to compute new lower bounds on the crossing number of the complete bipartite graph~$K_{m,n}$, extending a method from de Klerk et al.\ [SIAM J.\ Discrete Math.\ 20 (2006), 189--202] and the subsequent reduction by De Klerk, Pasechnik and Schrijver [Math.\ Prog.\ Ser.\ A and B, 109 (2007) 613--624]. 

We exploit the full symmetry of the problem using a novel decomposition technique. This results in a full block-diagonalization of the underlying matrix algebra, which we use to improve bounds on several concrete instances. Our results imply that $\cross(K_{10,n}) \geq  4.87057 n^2 - 10n$, $\cross(K_{11,n}) \geq 5.99939 n^2-12.5n$, $   \cross(K_{12,n}) \geq 7.25579 n^2 - 15n$,  $\cross(K_{13,n}) \geq    8.65675 n^2-18n$ for all~$n$.  The latter three bounds are computed using a new and well-performing relaxation of the original semidefinite programming bound. This new relaxation is obtained by only requiring one small matrix block to be positive semidefinite.
%\keywords{Crossing numbers \and Complete bipartite graph \and Semidefinite programming \and Symmetry reduction \and Block-diagonalization}
%\subclass{05C10 \and 68R10 \and  90C22 \and 05E10}

\section{Introduction}

Computing the crossing number $\cross(K_{m,n})$ of  the complete bipartite graph $K_{m,n}$ is a long-standing open problem, which goes back to Tur\'an in the 1940s. In 1956, Zarankiewicz~\cite{zarankiewicz} conjectured that $\cross(K_{m,n})= Z(m,n)$, where~$Z(m,n)$ is the \emph{Zarankiewicz number}
$$
Z(m,n):=\floor{\tfrac{m-1}{2}}\floor{\tfrac{m}{2}}\floor{\tfrac{n-1}{2}}\floor{\tfrac{n}{2}}=\floor{\tfrac{1}{4}(m-1)^2}\floor{\tfrac{1}{4}(n-1)^2}.
$$
Zarankiewicz claimed to have a proof for his conjecture, but this turned out to be false. The conjecture thus remains a notorious open problem. As \revise{Erd\H{o}s} and Guy~\cite{erdos} wrote in 1973: `Almost all questions that one can ask about crossing numbers remain unsolved', which is still true today. It is known that~$\cross(K_{m,n}) \leq Z(m,n)$, by exhibiting an explicit drawing of~$K_{m,n}$ in the plane with~$Z(m,n)$ crossings --- see Figure~\ref{fig:optimalexample} for an example. \revisefinal{The conjecture is proven for some small parameters: Kleitman~\cite{kleitman} proved it for $K_{m,n}$ with $m\leq 6$, and Woodall~\cite{woodall} proved it for $K_{7,7}$ and $K_{7,9}$.}% The smallest open cases are $K_{7,11}$ and $K_{9,9}$.}

In this paper, we use semidefinite programming and representation theory to prove the following lower bounds.
\begin{theorem}\label{theorem: bounds}
For all integers~$n$, 
\begin{align*}
    \cross(K_{10,n}) &\geq  4.87057 n^2 - 10n,
    \\   \cross(K_{11,n}) &\geq 5.99939 n^2-12.5n,
\\        \cross(K_{12,n}) &\geq 7.25579 n^2 - 15n,
 \\   \cross(K_{13,n}) &\geq    8.65675 n^2-18n.
\end{align*}
\end{theorem}
This theorem and Corollary~\ref{boundscorollary} below yield the best known lower bounds on all fixed $\cross(K_{m,n})$ with~$m,n \geq 10$. The best previously known lower bounds are $\cross(K_{m,n}) \geq  \tfrac{(m-1)m}{72} (3.86760n^2-8n)$ for~$m,n\geq 10$, cf.~\cite{regular}.  For an overview of known results regarding Zarankiewicz's conjecture, see the survey by Sz\'ekely~\cite{szekely}, or the survey about crossing numbers  by Schaefer~\cite{schaefer}. 
\begin{figure}[ht]\vspace{-10pt}\begin{center}
\begin{tikzpicture}[scale=0.55]
    % place nodes
    \node[draw, circle,fill=black,inner sep=0pt,minimum size=3pt] at (-3, 0)   (a1) {};
    \node[draw, circle,fill=black,inner sep=0pt,minimum size=3pt] at (-2, 0)   (a2) {};
    \node[draw, circle,fill=black,inner sep=0pt,minimum size=3pt] at (-1, 0)   (a3) {};
    \node[draw, circle,fill=black,inner sep=0pt,minimum size=3pt] at (0.75, 0)   (a4) {};
    \node[draw, circle,fill=black,inner sep=0pt,minimum size=3pt] at (1.5, 0)   (a5) {};
    \node[draw, circle,fill=black,inner sep=0pt,minimum size=3pt] at (2.25, 0)   (a6) {};
    \node[draw, circle,fill=black,inner sep=0pt,minimum size=3pt] at (3, 0)   (a7) {};

    \node[draw, circle,fill=black,inner sep=0pt,minimum size=3pt] at (0, -1.5)  (b1)     {};
    \node[draw, circle,fill=black,inner sep=0pt,minimum size=3pt] at (0, -.75)  (b2)     {};
    \node[draw, circle,fill=black,inner sep=0pt,minimum size=3pt] at (0, .75)  (b3)     {};
    \node[draw, circle,fill=black,inner sep=0pt,minimum size=3pt] at (0, 1.5)  (b4)     {};
    \node[draw, circle,fill=black,inner sep=0pt,minimum size=3pt] at (0, 2.25)  (b5)     {};
    % draw edges
    \draw[] (a1) node[above,xshift=1cm] {} -- (b1);
    \draw[] (a1) node[above,xshift=1cm] {} -- (b2);
    \draw[] (a1) node[above,xshift=1cm] {} -- (b3);
    \draw[] (a1) node[above,xshift=1cm] {} -- (b4);
    \draw[] (a1) node[above,xshift=1cm] {} -- (b5);

    \draw[] (a2) node[above,xshift=1cm] {} -- (b1);
    \draw[] (a2) node[above,xshift=1cm] {} -- (b2);
    \draw[] (a2) node[above,xshift=1cm] {} -- (b3);
    \draw[] (a2) node[above,xshift=1cm] {} -- (b4);
    \draw[] (a2) node[above,xshift=1cm] {} -- (b5);

    \draw[] (a3) node[above,xshift=1cm] {} -- (b1);
    \draw[] (a3) node[above,xshift=1cm] {} -- (b2);
    \draw[] (a3) node[above,xshift=1cm] {} -- (b3);
    \draw[] (a3) node[above,xshift=1cm] {} -- (b4);
    \draw[] (a3) node[above,xshift=1cm] {} -- (b5);

    \draw[] (a4) node[above,xshift=1cm] {} -- (b1);
    \draw[] (a4) node[above,xshift=1cm] {} -- (b2);
    \draw[] (a4) node[above,xshift=1cm] {} -- (b3);
    \draw[] (a4) node[above,xshift=1cm] {} -- (b4);
    \draw[] (a4) node[above,xshift=1cm] {} -- (b5);

    \draw[] (a5) node[above,xshift=1cm] {} -- (b1);
    \draw[] (a5) node[above,xshift=1cm] {} -- (b2);
    \draw[] (a5) node[above,xshift=1cm] {} -- (b3);
    \draw[] (a5) node[above,xshift=1cm] {} -- (b4);
    \draw[] (a5) node[above,xshift=1cm] {} -- (b5);

    \draw[] (a6) node[above,xshift=1cm] {} -- (b1);
    \draw[] (a6) node[above,xshift=1cm] {} -- (b2);
    \draw[] (a6) node[above,xshift=1cm] {} -- (b3);
    \draw[] (a6) node[above,xshift=1cm] {} -- (b4);
    \draw[] (a6) node[above,xshift=1cm] {} -- (b5);

    \draw[] (a7) node[above,xshift=1cm] {} -- (b1);
    \draw[] (a7) node[above,xshift=1cm] {} -- (b2);
    \draw[] (a7) node[above,xshift=1cm] {} -- (b3);
    \draw[] (a7) node[above,xshift=1cm] {} -- (b4);
    \draw[] (a7) node[above,xshift=1cm] {} -- (b5);
\end{tikzpicture}
\caption{Optimal drawing of~$K_{5,7}$. \label{fig:optimalexample}}
\end{center}\vspace{-20pt}
\end{figure}
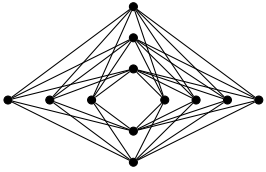

We now sketch how these lower bounds are derived. For~$m \in \N$, let~$Z_m$ be the set of permutations of~$[m]:=\{1,\ldots,m\}$ consisting of a single orbit, i.e., $Z_m$ is the set of all $m$-cycles from~$S_m$ and~$|Z_m|=(m-1)!$. Let~$K_{m,n}$ have colour classes~$\{1,\ldots,m\}$ and~$\{b_1,\ldots,b_n\}$. For any given drawing of~$K_{m,n}$ in the plane, define $\gamma(b_i)$ to be the cyclic permutation~$(1,i_2,\ldots,i_m) \in Z_m$ with the property that the edges leaving~$b_i$ in clockwise order go to~$1,i_2,\ldots,i_m$. 

\begin{table}[ht]\small\centering
    \begin{tabular}{cS[table-format=4.0]>{\bfseries}S[table-format=4.0]S[table-format=4.0]}\toprule
% $n$& \multicolumn{1}{>{\raggedright\arraybackslash}b{9.5mm}|}{\scriptsize best lower bound previously known}  & \multicolumn{1}{>{\raggedright\arraybackslash}b{8.05mm}|}{\small new lower bound} & $Z(n,n)$  \\\midrule
$n$& {\begin{tabular}{@{}c@{}}best previously\\ known lower bound\end{tabular}}  & {\small new lower bound} & {$Z(n,n)$}  \\\midrule
         10 &  384 & 388    & 400  \\
         11   & 581 & 589 & 625  \\    %0.8726 =8beta_k/ (k(k-1))
         12  & 846& 865 & 900 \\ % 0.8794 =8beta_k/ (k(k-1)) (rounded down on 4 decimals)
      13 &  1192&1229  & 1296 \\\bottomrule 
    \end{tabular}\caption{\footnotesize Some of our new lower   bounds on~$\cross(K_{n,n})$. The previously best known lower bounds follow from~\cite{regular}.}\end{table}

Let~$Q$ be the~$Z_m \times Z_m$ matrix with for any~$\sigma,\tau \in Z_m$,  the entry $Q_{\sigma,\tau}$ is equal to the minimum number of crossings in any drawing of~$K_{m,2}$ with~$\gamma(b_1) =\sigma$ and~\revise{$\gamma(b_2)=\tau$}. 
This matrix was defined in~\cite{klerkcrossing} and later also used in~\cite{regular}. An algorithm to compute~$Q_{\sigma,\tau}$ was used by Kleitman~\cite{kleitman} and more details were described by Woodall~\cite{woodall}. For example, $Q_{\sigma,\sigma}=\floor{ \tfrac{1}{4} (m-1)^2}$ for all~$\sigma \in Z_m$. Let~$\bm{1} \in \R^{Z_m}$ denote the all-ones vector. Consider the following quadratic program.
\begin{align}\label{mainqp}
  q_m := \min \left\{x\T Q x \, | \, x \in \R^{Z_m}_{\geq 0}, \, x\T \bm{1} =1  \right\}. 
\end{align}\vspace{-3pt}
\begin{theorem}[De Klerk et al.~\cite{klerkcrossing}]\label{klerktheorem}
$\cross(K_{m,n})\geq \tfrac{1}{2}n^2q_m - \tfrac{1}{2} n \floor{ \tfrac{1}{4} (m-1)^2}$ for all~$m,n$.
\end{theorem}
\proof 
Suppose a drawing of~$K_{m,n}$ with~$\cross(K_{m,n})$ crossings is given. For each~$\sigma \in Z_m$, let $c_{\sigma}$ be the number of vertices~$b_i$ with~$\gamma(b_i)=\sigma$. We view~$c$ as a vector in~$\R^{Z_m}$ and define $x := n^{-1} c$.
Then~$x$ satisfies the conditions in~\eqref{mainqp}, so~$q_m \leq x\T Qx$. For~$i,j \in [n]$ let~$d_{i,j}$ be the number of crossings of edges leaving~$b_i$ with edges leaving $b_j$. By definition of~$Q$, if~$i\neq j$, then~$d_{i,j} \geq Q_{\gamma(b_i),\gamma(b_j)}$.  This implies \vspace{-10pt}
\begin{align*}
  \tfrac{1}{2} n^2 q_m &\leq \tfrac{1}{2} n^2 x\T Qx = \tfrac{1}{2}c\T Q c =  \tfrac{1}{2}\sum_{i,j=1}^n Q_{\gamma(b_i),\gamma(b_j)}  \leq  \sum_{\substack{i< j}} d_{i,j} + \tfrac{1}{2}\sum_{i=1}^n Q_{\gamma(b_i),\gamma(b_i)} 
   \\[0pt]&\revise{\leq} \cross(K_{m,n}) + \tfrac{1}{2}n\floor{\tfrac{1}{4}(m-1)^2},
\end{align*}
where the last inequality follows from $Q_{\sigma,\sigma}=\floor{ \tfrac{1}{4} (m-1)^2}$ for all~$\sigma \in Z_m$. \revise{(In fact, the last inequality is an equality as one may assume that in an optimal drawing edges incident to a common vertex do not cross, cf.~\cite{GJ83}.)} 
\endproof
The following semidefinite programming parameter~$\alpha_m$ is a lower bound on~$q_m$.
\begin{align}\label{mainsdp}
    \alpha_m := \min \left\{\langle Q, X\rangle \, | \, X \in \R^{Z_m \times Z_m}_{\geq 0}, \, \langle J,X \rangle =1, \, X \succeq 0  \right\}.
\end{align}
Here $X\succeq 0$ means `$X$ symmetric and positive semidefinite'. It is clear that~$q_m \geq \alpha_m$, as any feasible~$x$ for~$q_m$ gives a feasible~$X=xx\T$ for~$\alpha_m$ with the same objective value. The values~$\alpha_m$ for~$m\leq 7$ were computed by De Klerk, Maharry,  Pasechnik, Richter, and Salazar~\cite{klerkcrossing}. Dobre and Vera~\cite{dobrevera} computed a better lower bound on~$q_7$ using semidefinite approximations of the copositive cone. The values~$\alpha_8$ and~$\alpha_9$ were computed by De Klerk, Pasechnik and Schrijver~\cite{regular}, who used the regular $*$-representation to reduce the semidefinite programs in size. The regular $*$-representation found several applications (see, e.g.,~\cite{laurent} for an application in coding theory). In this paper, we show how a full block-diagonalization can be obtained\revise{, where we exploit properties of the representation theory of the symmetric group for computational efficiency.} This allows us to compute the value~$\alpha_{10}$.

\revise{A full symmetry reduction for computing~$\alpha_m$ has been developed before by Hymabaccus and Pasechnik~\cite{HP20}. Their method can be used to decompose representations of finite groups exactly. Due to the generic nature of their algorithm, they work with representation matrices instead of vectors in the representative sets. This costs a lot of memory (and time), so they only reach~$\alpha_7$ with their method.  In the crossing number case, the coefficients in their block-diagonalization  contain irrational numbers, potentially leading to rounding issues in floating-point computations. An advantage of our approach is, apart from being more memory and time efficient, that it results in an exact block-diagonalization with integer coefficients.}

Our symmetry reduction consists of three steps.  First, we use classical representation theory of the symmetric group~$S_m$ to decompose a well-known permutation module. Secondly, we we use an elementary but crucial observation given in Proposition~\ref{representativeprop}, to transform this decomposition into a decomposition of~$\R^{Z_m}$ as $S_m$-module. Proposition~\ref{representativeprop} has potential for a wide array of applications, for example, it can also be directly applied to a problem in coding theory, which we describe in Remark~\ref{codingremark} below. The third and final step in our block-diagonalization takes into account a separate $\{\pm 1 \}$-action, in Proposition~\ref{s2proposition}. 

Inspired by our symmetry reduction, we also formulate a new relaxation of~$\alpha_m$, which we call~$\beta_m$. The value~$\beta_m$ is obtained from~$\eqref{mainsdp}$ by only requiring that one specified block, which is described in Section~\ref{sec: betak} below,  in the block-diagonalization is positive semidefinite instead of the full matrix~$X$. So~$\beta_m \leq \alpha_m$, and our experiments show that the new bound~$\beta_m$ is remarkably close to~$\alpha_m$.  We give a combinatorial desciption of the vectors which underly the block-diagonalization of~$\beta_m$ in Proposition \ref{betacombinatorialproposition}. Also, we compute the value~$\beta_m$ for~$m\leq 13$. The values are provided in Table~\ref{table: boundstable}. Inserting our newly computed values $\alpha_{10}$, $\beta_{11}$, $\beta_{12}$, $\beta_{13}$  in Theorem~\ref{klerktheorem} instead of~$q_k$ (using the fact that $\beta_k \leq \alpha_k \leq q_k$), we directly obtain our new bounds in Theorem~\ref{theorem: bounds}.
\vspace{-8pt}
% \sisetup{round-mode=places,round-precision=3}
\begin{table}[ht]\centering \small
    \begin{tabular}{S[table-format=2.0]S[table-format=1.10]S[table-format=1.4]S[table-format=2.10]S[table-format=1.4]}\toprule
$m$& $\alpha_m$ &  $ \frac{8\alpha_m}{k(k-1)}$  & $\beta_m$  &$ \frac{8\beta_m}{m(m-1)}$\\\midrule
         4 & 1.0000000000 & 0.6667  & 1.0000000000 & 0.6667 \\
         5 & 1.9472135954 & 0.7789 & 1.9270509831  & 0.7708 \\
         6 & 2.9519183588 & 0.7872 & 2.9519183588  & 0.7872 \\
         7 & 4.3593154948 & 0.8303 &   4.3107391257 &  0.8210 \\
         8 & 5.8599856417   & 0.8371 & 5.8284271247 &  0.8326 \\
         9 & 7.7352125975 &  0.8595 & 7.6527560430 &   0.8503\\
         10 &  $\bm{9.7411403685}$   & $\bm{0.8659}$ & 9.6866252078    &0.8610 \\
         11 &  & & $\bm{11.9987919703}$ & $\bm{0.8726}$ \\    %0.8726 =8beta_k/ (k(k-1))
         12 & & & $\bm{14.5115811776}$ & $\bm{0.8794}$\\ % 0.8794 =8beta_k/ (k(k-1)) (rounded down on 4 decimals)
      13 & & & $\bm{17.3135089904}$ & $\bm{0.8878}$\\\bottomrule 
    \end{tabular}
    \caption{The full semidefinite bound~$\alpha_m$ from~\eqref{mainsdp} and our  relaxation $\beta_m$ which is described in Section~\ref{sec: betak}. We solved the SDPs with multiple precision versions of SDPA~\cite{nakata}, \revise{and then rounded the dual solutions to rational feasible dual solutions, see Section \ref{sec: verify bounds}.}\label{table: boundstable} }
\end{table}\vspace{-25pt}
\subsection{Outline of the paper}\vspace{-5pt}
In Section~\ref{sec: derived} we explain the consequences of Theorem~\ref{theorem: bounds}: we investigate to which bounds it leads and relate these bounds to the literature. In Section~\ref{sec: symmetry} we explain how the symmetry can be used to significantly reduce the problem: we develop a full block-diagonalization. To do this, we use representation theory from the symmetric group and linear algebra. After that, we explain in Section~\ref{sec: betak} how our new relaxation~$\beta_m$ of~$\alpha_m$ is defined, which is inspired by the symmetry reduction. We give a combinatorial desciption of the vectors which underly the block-diagonalization of~$\beta_m$.
Finally, in Section~\ref{sec: computation} we give details about our computations. Here we explain how~$\beta_m$ can be computed in practice: using the dual description in combination with an iterative procedure, we are able to compute~$\beta_m$ for~$m \leq 13$ up to high precision on a desktop computer.  %we formulate the dual semidefinite program which has many linear constraints and only one small matrix block which is required to be positive semidefinite. We solve the semidefinite program without the linear constraints, and iteratively add the most violated constraint, until no constraints are violated anymore. In this way, we are able to compute~$\beta_k$ for~$k \leq 13$ up to high precision on a desktop computer.

\vspace{-10pt}
\section{Derived lower bounds \label{sec: derived}}\vspace{-5pt}
Suppose that~$2 \leq k \leq m$ and that~$n \in \N$. There are~$\binom{m}{k}$ distinct copies of~$K_{k,n}$ in~$K_{m,n}$, and in any drawing of~$K_{m,n}$, each crossing appears in~$\binom{m-2}{k-2}$ distinct copies of~$K_{k,n}$. This implies that
\begin{align} \label{eq: doublecount}
\cross(K_{m,n}) \geq \frac{\cross(K_{k,n}) \binom{m}{k}}{\binom{m-2}{k-2}} = \frac{\cross(K_{k,n}) \cdot m(m-1)}{k(k-1)}.
\end{align}
So any lower bound on~$q_k$ gives lower bounds on $\cross(K_{m,n}) $  for all~$m \geq k$ and all~$n$. Combining~\eqref{eq: doublecount} with our new lower bounds~$\alpha_{10}$, $\beta_{11}$, $\beta_{12}$, $\beta_{13}$ presented in Table~\ref{table: boundstable} gives the following.
\begin{corollary}\label{boundscorollary} For all integers $n$ we have
\begin{align*}
    \text{for all~$m \geq 10$},\,\, \cross(K_{m,n}) &\geq  0.0541m(m-1)n^2 - \tfrac{1}{9}m(m-1)n,
    \\         \text{for all~$m \geq 11$},\,\,   \cross(K_{m,n}) &\geq 0.0545m(m-1)n^2 - \tfrac{5}{44}m(m-1)n,
\\             \text{for all~$m \geq 12$},\,\,    \cross(K_{m,n}) &\geq 0.0549m(m-1)n^2 - \tfrac{5}{44}m(m-1)n,
 \\          \text{for all~$m \geq 13$},\,\,  \cross(K_{m,n}) &\geq  0.0554m(m-1)n^2 - \tfrac{3}{26}m(m-1)n.
\end{align*}
\end{corollary}
\proof
By Theorem~\ref{klerktheorem}, we have $\cross(K_{k,n})\geq \tfrac{1}{2}n^2q_k - \tfrac{1}{2} n \floor{ \tfrac{1}{4} (k-1)^2}$ for all~$k,n$. We also have~$q_k \geq \alpha_k \geq \beta_k$ for all~$k$, hence the inequality holds upon replacing~$q_k$ by~$\alpha_k$ or~$\beta_k$. Combining this equation  with our computed values  $\alpha_{10}$, $\beta_{11}$, $\beta_{12}$, $\beta_{13}$ results in lower bounds on~$\cross(K_{10,n})$, $\cross(K_{11,n})$, $\cross(K_{12,n})$ and~$\cross(K_{13,n})$, respectively. Inserting these lower bounds in equation~\eqref{eq: doublecount} for~$\cross(K_{k,n})$ yields the corollary. 
\endproof

The lower bounds also allow to give statements about limits, using the following lemma.

\begin{lemma}[De Klerk et al.~\cite{klerkcrossing}]
$\displaystyle{\lim_{n \to \infty} \frac{\cross(K_{m,n})}{Z(m,n)} \geq \frac{8 q_k}{k(k-1)} \frac{m}{m-1}}$  for all~$k 
\leq m$.
\end{lemma}
\begin{proof} \revise{First, note that the limit exists: the sequence~$(\text{cr}(K_{m,n})/\tbinom{n}{2})_{n \in \N}$ for fixed~$m$ is nondecreasing (by the same calculation as in~\eqref{eq: doublecount} but now applied to~$n$ instead of~$m$) and bounded (using~$\text{cr}(K_{m,n}) \leq Z_{m,n}$), hence has a limit.  For fixed~$m$, both $Z_{m,n}$ and~$\tbinom{n}{2}$ grow quadratically in~$n$, so the limit $\frac{\cross(K_{m,n})}{Z(m,n)})_{n \in \N}$ exists too.} 
\revise{The lemma now follows from an elementary calculation} using the bounds previously given. By Theorem~\ref{klerktheorem}, we have $\cross(K_{k,n})\geq \tfrac{1}{2}n^2q_k - \tfrac{1}{2} n \floor{ \tfrac{1}{4} (k-1)^2}$ for all~$k,n$. Now, we use~\eqref{eq: doublecount} and find, for~$m\geq k$:
\begin{align*}
\hspace{45pt}\lim_{n \to \infty} \frac{\cross(K_{m,n})}{Z(m,n)} &\geq\lim_{n \to \infty} \frac{m(m-1)(\tfrac{1}{2}n^2q_k - \tfrac{1}{2} n \floor{ \tfrac{1}{4} (k-1)^2})}{k(k-1)Z_{m,n}} 
\\&= \lim_{n \to \infty} \frac{m(m-1)(\tfrac{1}{2}n^2q_k - \tfrac{1}{2} n \floor{ \tfrac{1}{4} (k-1)^2})}{k(k-1)\floor{\tfrac{1}{4}(m-1)^2}\floor{\tfrac{1}{4}(n-1)^2}}
\\& =  \frac{2 q_k}{k(k-1)} \frac{m(m-1)}{\floor{\tfrac{1}{4}(m-1)^2}} 
\geq   \frac{8 q_k}{k(k-1)} \frac{m}{m-1}.  \hspace{44pt}   %qedhere does not work
\end{align*} 
\end{proof}
As~$q_k \geq \alpha_k \geq \beta_k$, the lemma also holds upon replacing~$q_k$ by~$\alpha_k$ or~$\beta_k$. So our computed values  $\alpha_{10}$, $\beta_{11}$, $\beta_{12}$, $\beta_{13}$ give asymptotic lower bounds on $\lim_{n \to \infty} \frac{\cross(K_{m,n})}{Z(m,n)} $ for~$m \geq k$. In the following lemma, we provide the lower bound for~$m \geq 13$, using our computed value~$\beta_{13}$. The lower bounds for~$m=10,11,12$ are displayed in Table~\ref{table: boundstable}. 
\begin{corollary}For all $m \geq 13$,
$
\displaystyle{\lim_{n \to \infty} \frac{\cross(K_{m,n})}{Z(m,n)} \geq 0.8878 \tfrac{m}{m-1}.}
$
\end{corollary}
A direct result of this corollary is
\begin{align}\label{asymptoticbound}
\displaystyle{\lim_{n \to \infty} \frac{\cross(K_{n,n})}{Z(n,n)} \geq 0.8878.}
\end{align}
The previously best known published lower bound on $ \lim_{n \to \infty} \frac{\cross(K_{n,n})}{Z(n,n)} $ is $0.8594$ (which follows using~$\alpha_9$), cf. De Klerk et al.~\cite{regular}. Norin and Zwols obtained a lower bound of~$0.905$  using flag algebras which they presented at a workshop~\cite{norin}. Recently, Balogh, Lidick\'y, Norin, Pfender, Salazar, and Spiro obtained a lower bound of~$0.9118$, also using flag algebras~\cite{flagcrossing}. These flag algebra bounds are stronger than our bound in~\eqref{asymptoticbound}. In~\cite{hill}, Balogh, Lidick\'y and Salazar prove very strong asymptotic lower bounds on the crossing number of the complete graph using flag algebras. %It might be possible to improve upon~\eqref{asymptoticbound} bound by using a similar approach considering high levels in the flag algebra hierarchy. 

However, in order to prove asymptotic bounds it is also worthwhile to further investigate the quadratic programming hierarchy from De Klerk et al.~\cite{klerkcrossing} which we consider in this paper. One might hope to prove lower bounds~$t_k$ on~$\alpha_k$ such that~$8t_k /(k(k-1)) \to 1$ as~$k \to \infty$, thereby proving $\lim_{n \to \infty} \frac{\cross(K_{n,n})}{Z(n,n)}=1$, i.e., asymptotically proving Zarankiewicz' conjecture. Figure~\ref{figure: graph} gives rise to the question whether $8\beta_k /(k(k-1)) \to 1$ as~$k \to \infty$. 

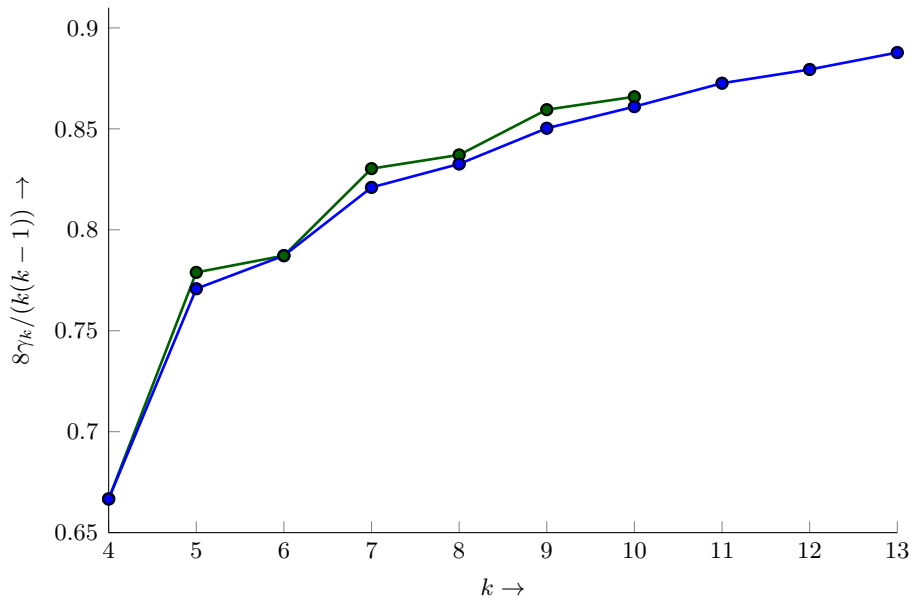
\begin{figure}[H]
\centering
\pgfplotsset{xmin=4, xmax=13, ymin=0.65, ymax=0.91}   
\begin{tikzpicture}
\begin{axis}[domain=4:10, restrict y to domain=0.6:1,
width={0.7\textwidth}, height={0.5\textwidth},
axis x line*=bottom,axis y line*=left,
    /pgf/number format/.cd,tick label style={font=\footnotesize},
 %       use comma,
    ylabel = {\footnotesize $8 \gamma_k / (k(k-1))$ $\Large\rightarrow$},
ylabel style={yshift=0cm,xshift=0cm},
xlabel={$k \rightarrow$},
xlabel style={yshift=0cm,xshift=0cm,font=\footnotesize},   %was yshift.25, xshift.4=3
/pgfplots/scaled ticks=false,
      1000 sep={},
precision=6,
ylabel style={yshift=0cm},%      ylabel={\color{firebrick} $2^n$},
   %   yticklabels=\empty,    xticklabels=\empty, xtick style={draw=none},ytick style={draw=none},%was not empty
 every axis plot/.append style={thick}]
  \addplot+ [mark=*,donkergroen2, line width={1}] table {
%2  0
%3  .6667
4  .6667
5  .7789
6  .7872
7  .8303
8  .8371
9  .8595
10 .8659
};
  \addplot+ [mark=*,donkerblauw, line width={1}] table {
%2  0
%3  .6667
4  .6667
5  .7708
6  .7872
7  .8210
8  .8326
9  .8503
10 .8610
11 .8726
12 .8794
13 .8878
};
\addplot[mark=*,donkergroen2, mark size={2.0 pt}, mark options={color=black,fill={donkergroen2}, line width={0.75}}] coordinates {(4,0.6667)};
\addplot[mark=*,donkergroen2, mark size={2.0 pt}, mark options={color=black,fill={donkergroen2}, line width={0.75}}] coordinates {(5,0.7789)};
\addplot[mark=*,donkergroen2, mark size={2.0 pt}, mark options={color=black,fill={donkergroen2}, line width={0.75}}] coordinates {(6,0.7872)};
\addplot[mark=*,donkergroen2, mark size={2.0 pt}, mark options={color=black,fill={donkergroen2}, line width={0.75}}] coordinates {(7,0.8303)};
\addplot[mark=*,donkergroen2, mark size={2.0 pt}, mark options={color=black,fill={donkergroen2}, line width={0.75}}] coordinates {(8,0.8371)};
\addplot[mark=*,donkergroen2, mark size={2.0 pt}, mark options={color=black,fill={donkergroen2}, line width={0.75}}] coordinates {(9,0.8595)};
\addplot[mark=*,donkergroen2, mark size={2.0 pt}, mark options={color=black,fill={donkergroen2}, line width={0.75}}] coordinates {(10,0.8659)};
\addplot[mark=*,donkerblauw, mark size={2.0 pt}, mark options={color=black,fill={donkerblauw}, line width={0.75}}] coordinates {(4,0.6667)};
\addplot[mark=*,donkerblauw, mark size={2.0 pt}, mark options={color=black,fill={donkerblauw}, line width={0.75}}] coordinates {(5,0.7708)};
\addplot[mark=*,donkerblauw, mark size={2.0 pt}, mark options={color=black,fill={donkerblauw}, line width={0.75}}] coordinates {(7,0.8210)};
\addplot[mark=*,donkerblauw, mark size={2.0 pt}, mark options={color=black,fill={donkerblauw}, line width={0.75}}] coordinates {(8,0.8326)};
\addplot[mark=*,donkerblauw, mark size={2.0 pt}, mark options={color=black,fill={donkerblauw}, line width={0.75}}] coordinates {(9,0.8503)};
\addplot[mark=*,donkerblauw, mark size={2.0 pt}, mark options={color=black,fill={donkerblauw}, line width={0.75}}] coordinates {(10,0.8610)};
\addplot[mark=*,donkerblauw, mark size={2.0 pt}, mark options={color=black,fill={donkerblauw}, line width={0.75}}] coordinates {(11,0.8726)};
\addplot[mark=*,donkerblauw, mark size={2.0 pt}, mark options={color=black,fill={donkerblauw}, line width={0.75}}] coordinates {(12,0.8794)};
\addplot[mark=*,donkerblauw, mark size={2.0 pt}, mark options={color=black,fill={donkerblauw}, line width={0.75}}] coordinates {(13,0.8878)};
\end{axis}
% \node [right] at (5, 101) {\scriptsize$10^6$};
\end{tikzpicture}
\caption{We have the lower bound $\lim_{n \to \infty} \text{cr}(K_{m,n})/Z(m,n) \geq (8\gamma_k/(k(k-1))) m/(m-1)$ for each~$m\geq k$ and~$\gamma_k \in \{\alpha_k,\beta_k\}$. The values $8\alpha_k/(k(k-1))$ are plotted in green and the values $8\beta_k/(k(k-1))$ are plotted in blue.   }\label{figure: graph}
\end{figure}

In Figure~\ref{figure: graph}, the increases are larger for odd~$k$ than for even~$k$, a trend which was already noted in~\cite{regular}. We now see that this trend continues for some larger~$k$. As noted in~\cite{regular}, this is reminiscent of the fact that Zarankiewicz's conjecture holds for $K_{2m,n}$ if it holds for $K_{2m-1,n}$.

\section{Exploiting the symmetry of the problem \label{sec: symmetry}}

Recall that~$Z_m$ is the set of permutations of~$[m]$ consisting of a single orbit, i.e., $Z_m$ is the set of all $m$-cycles from~$S_m$ and~$|Z_m|=(m-1)!$. The group~$G_m:=S_m \times \{ \pm 1\}$ acts on~$Z_m$ via
$$
(\pi,\varepsilon) \cdot \sigma  = \pi \sigma^{\varepsilon} \pi^{-1},
$$
for~$\sigma \in Z_m$, $(\pi,\varepsilon) \in G_m=  S_m \times  \{ \pm 1 \}$. If~$X$ is any optimum solution for the program~\eqref{mainsdp} defining~$\alpha_m$, also~$g \cdot X$ with~$(g \cdot X)_{\sigma, \tau} = X_{g \cdot \sigma, g \cdot \tau}$  is feasible for all~$g \in G_m$: the matrix~$g \cdot X$ is obtained from~$X$ by simultaneously permuting rows and columns, which preserves positive semidefiniteness, entrywise nonnegativeness and the total sum of the entries. Moreover, the objective values corresponding to~$X$ and~$g \cdot X$ are the same. Indeed, as~$g \cdot Q = Q$ for all~$g \in G_m$, one has~$\langle Q, X \rangle = \langle g \cdot Q, g\cdot X \rangle =\langle  Q, g\cdot X \rangle  $. As~$G_m$ is a finite group and the feasible region in~\eqref{mainsdp} is convex, we can replace any optimum solution~$X$ by the group average~$(1/|G_m|)\sum_{g \in G_m} g \cdot X$ to obtain a $G_m$-invariant optimum solution. So we may assume our optimum solution is~$G_m$-invariant, i.e., its entries are constant on~$G_m$-orbits of~$Z_m \times Z_m$. Hence the number of variables is \revise{the cardinality of}~$\Omega_m := (Z_m \times Z_m)/G_m$ (where~$G_m$ acts on both copies of~$Z_m$ simultaneously). \revise{The set~$\Omega_m$ is also known as the set of \emph{orbitals} of~$G_m$ acting on~$Z_m$, and $\lvert\Omega_m\rvert$ as the rank of the action of $G_m$, see, e.g.,~\cite{Cameron1999}.} The number of variables can be reduced further since~$X$ is symmetric, so the value of~$X$ on the orbit of~$(\sigma,\tau)$ is the same as the value of~$X$ on the orbit of~$(\tau,\sigma)$. We write~$\Omega_m'$ to be the collection of these `symmetric' $G_m$-orbits on~$Z_m \times Z_m$, in which orbits of $(\sigma,\tau)$ and~$(\tau,\sigma)$ are identified. This gives a significant reduction in the number of variables which was already used in~\cite{klerkcrossing}. 

It is also possible to reduce the size of the matrix~$X$ in the semidefinite programming formulation. In~\cite{regular}, the regular $*$-representation was used, which reduced checking whether a \revise{$G_m$}-invariant matrix~$X$ is positive semidefinite into checking whether a matrix of order~$|\Omega_m| \times |\Omega_m|$ is positive semidefinite. In this paper, we will reduce the matrix~$X$ further, by developing a full \emph{block-diagonalization}. For any finite group~$G$ acting on a vector space~$V$, we write~$V^G$ for the subspace of~$V$ of~$G$-invariant elements. The block-diagonalization is a bijective linear map
\begin{align}
 \Phi \, : \, \left(\C^{Z_m \times Z_m}\right)^{G_m} \to \bigoplus_{i=1}^k \C^{m_i \times m_i},
\end{align}
for some integer~$k$ and integers~$m_i$ for~$i \in [k]$, such that~$X \in \left(\C^{Z_m \times Z_m}\right)^{G_m} $ is positive semidefinite if and only if~$\Phi(X)$ is positive semidefinite. It has the property that~$\sum_{i=1}^k m_i^2 = |(Z_m \times Z_m)/G_m|= |\Omega_m|$, which is considerably smaller than~$|Z_m|^2$.

\subsection{Preliminaries on representation theory}\label{prelimsec}
We here describe the preliminaries on representation theory which we will use throughout the paper, based on a combination of the notation and definitions used in references~\cite{ brosch,gatermann, LPS,sagan}.
If~$G$ is a finite group acting on a complex vector space~$V$ of finite dimension, $V$ is called a \emph{$G$-module}. Any $G$-invariant subspace of~$V$ is called a \emph{submodule}. If $V$ and~$W$ are $G$-modules, a  \emph{$G$-homomorphism} is a  linear map~$\psi \, : \, V \to W$ with $g \cdot \psi(v) = \psi(g \cdot v)$ for all~$g \in G$ and~$v \in V$.  The modules~$V$ and~$W$ are \emph{equivalent} (or \emph{$G$-isomorphic}) if there is a bijective $G$-homomorphism (called a \emph{$G$-isomorphism}) from $V$ to $W$. A $G$-module $V$ is \emph{irreducible} if $V \neq 0$ and its only nonzero submodule is~$V$. The \emph{centralizer algebra} of the action of~$G$ on~$V$,  denoted by~$\text{End}_G(V)$,  is the algebra of $G$-homomorphisms $V \to V$.

Let again~$G$ be a finite group acting on a complex finite dimensional vector space~$V$. Then one can decompose
$
V = \bigoplus_{i=1}^k \bigoplus_{j=1}^{m_i} V_{i,j},
$
for some unique number $k$ and numbers~$m_1,\ldots,m_k$ (which are unique up to permutation), such that the~$V_{i,j}$ are irreducible submodules of~$V$ with the property that~$V_{i,j}$ is isomorphic to~$V_{i',j'}$ if and only if~$i=i'$. 

\begin{definition}[Representative set] For each~$i \leq k$ and~$j \leq m_i$ let~$u_{i,j} \in V_{i,j}$ be a nonzero vector, such that for each~$i \leq k$ and~$j,j' \leq m_i$ there exists a $G$-isomorphism from~$V_{i,j}$ to~$V_{i,j'}$ which maps~$u_{i,j}$ to~$u_{i,j'}$. Define, for each~$i \leq k$, the tuple $U_i := (u_{i,1},\ldots, u_{i,m_i})$. \revise{Call} any set~$\{U_1,\ldots,U_k\}$ obtained in this way a \emph{representative set} for the action of~$G$ on~$V$.
\end{definition}

We can view the $U_i$ as matrices by seeing the vectors $u_{i,j}$ (for $j =1,\ldots,m_i$) as its columns, and we will do so depending on the context.

 \revise{The space~$V$} has a $G$-invariant inner product~$\langle , \rangle$. Let~$\{U_1,\ldots,U_k\}$ be any representative set for the action of~$G$ on~$V$, and define the map $ \Phi \colon \text{End}_G(V) \to \bigoplus_{i=1}^k \C^{m_i \times m_i}$ which maps $ A \mapsto \bigoplus_{i=1}^k \left( \langle A u_{i,j'} , u_{i,j}  \rangle \right)_{j,j' =1}^{m_i}$. This map is linear and bijective, and it has the property that $A \succeq 0$ if and only if~$\Phi(A) \succeq 0$. This follows from classical representation theory. For a proof, see e.g.,~\cite[Proposition 2.4.4]{polak}. We apply it to the following. Suppose that~$G$ is a finite group acting on a finite set~$Z$, hence on the vector space~$V:=\C^Z$. Then $\text{End}_G(V)$ can be naturally identified with $\left(\C^{Z \times Z}\right)^G$, and the map~$\Phi$ becomes
\begin{align} 
    \Phi \colon(\C^{Z \times Z})^{G} \to \bigoplus_{i=1}^k \C^{m_i \times m_i} \,\, \text{ with } \,\, A \mapsto \bigoplus_{i=1}^k U_i^* A U_i.
\end{align}
It will turn out that all representative sets in this paper consist of real matrices. So we  can replace~$\C$ by~$\R$ in the above equation: $\Phi$ is a linear bijective map~$(\R^{Z \times Z})^{G} \to \bigoplus_{i=1}^k \R^{m_i \times m_i}$ such that~$A \succeq 0$ if and only if~$\Phi(A)\succeq 0$ for all~$A \in (\R^{Z \times Z})^{G}$.

\paragraph{Representation theory of the symmetric group.}
A \emph{partition}~$\lambda$  of $n$ is a sequence of integers~$\lambda_1 \geq \ldots \geq \lambda_h >0$ with~$\lambda_1 + \ldots +\lambda_h = n$ for some $h \in \N$ which is called the \emph{height} of~$\lambda$. We write~$\lambda \vdash n$ to denote that~$\lambda$ is a partition of~$n$. The \emph{(Young) shape} of~$\lambda \vdash n$ is an array consisting of~$n$ boxes divided into~$h$ rows where for each $1 \leq i \leq h$, the~$i$-th row contains~$\lambda_i$ boxes. As an example, consider the shape corresponding to~$(4,1,1) \vdash 6$:
$$
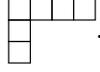
\ytableausetup{centertableaux,boxsize=0.7em}
\begin{ytableau}
\ & \ & \  &\  \\
\  \\
\ 
\end{ytableau}.
$$
A \emph{Young tableau of shape~$\lambda$} is a filling~$\tau$ of the boxes of the Young shape~$\lambda$ with the integers~$1,\ldots,n$, where each number appears once.  Two Young tableaux~$t$, $t'$ of shape~$\lambda \vdash n$ are \emph{(row) equivalent}, written as~$t \sim t'$ if corresponding rows of the two tableaux contain the same elements. A \emph{tabloid} of shape~$\lambda$ is an equivalence class of tableaux: $\{t\} = \{ t ' \, : \, t ' \sim t\}$. We denote a tabloid by an array with lines between the rows, e.g.,
$$
\left\{ \ytableausetup{centertableaux,boxsize=1em,notabloids}
\begin{ytableau}
1 & 3  \\
2  \\
\end{ytableau}, \begin{ytableau}
3 & 1  \\
2  \\
\end{ytableau}\right\}={\ytableausetup{centertableaux,boxsize=1em,tabloids}
\begin{ytableau}
1 & 3  \\
2  \\
\end{ytableau} }.
$$
 Any permutation~$\pi \in S_n$ acts on a tableau~$t=t_{i,j}$ by acting on its content, i.e., $\pi t = (\pi(t_{i,j}))$.  The \emph{column stabilizer} $C_{t}$ of a tableau $t$ is the subgroup of~$S_n$ which leaves the columns of~$t$ invariant. The action of~$\pi \in S_n$ on a tableau~$t$ extends to a well-defined action on tabloids via~$\pi \{t\} = \{ \pi t\}$. For each~$\lambda \vdash n$ the \emph{permutation module} $M^{\lambda}$ \emph{corresponding to $\lambda$} is defined as
$$
M^{\lambda} = \C \{ \{t_1\},\ldots,\{t_k\}\},
$$
where~$\{t_1\}, \ldots, \{t_k\}$ is a complete set of~$\lambda$-tabloids. For any tableau~$t$, the associated \emph{polytabloid} is
$
e_{t}:= \sum_{c \in C_{t}} \text{sgn}(c) c \{t\}.
$
The \emph{Specht module $S^{\lambda}$ corresponding to~$\lambda$} is the submodule of~$M^{\lambda}$ spanned by the polytabloids $e_{t}$, where $t$ is a tableau of shape~$\lambda$. The module $S^{\lambda}$ is irreducible, and it is generated by any given polytabloid: $S^{\lambda} = \C S_n\cdot e_{t}$ for any fixed $\lambda$-tableau $t$. 

A \emph{generalized} Young tableau of shape~$\lambda \vdash n$ is a (Young) shape filled with integers, where we allow repeated entries. Depending on the context, we often omit the word `generalized'. A generalized Young tableau is \emph{standard} if its rows and columns are strictly increasing, and \emph{semistandard} if its rows are nondecreasing and its columns are strictly increasing. We say that a generalized tableau of shape~$\lambda \vdash n$ has \emph{content}~$\mu =(\mu_1,\ldots,\mu_h) \vdash n$ if it contains~$\mu_i$ times the integer~$i$, for all~$1 \leq i \leq h$. If~$T$ is any tableau of shape~$\lambda$ and content~$\mu$, the map 
\begin{align*}
\vartheta_T: M^{\lambda} &\to M^{\mu},\\
\{t\} &\mapsto \sum_{T' \sim T}t[T'] \quad \revise{\text{(extended linearly to $M^\lambda$)}},
\end{align*}
\revise{where $\{t\}$ is any tabloid in $M^\lambda$, and} where
\[t[T'] := \{\text{tableau with entry $t_{i,j}$ in its $T_{i,j}'$-th row}\},\]
is an $S_n$-homomorphism. Moreover, a basis of $\text{Hom}(S^{\lambda}, M^{\mu})$ is given by (cf.\ Sagan~\cite{sagan})
$$
\{\vartheta_T \, | \, T \text{ semistandard of shape~$\lambda$ and content~$\mu$}\}.
$$
\revise{Unless specified otherwise, from now on we assume that~$t$  is the $\lambda$-tableau containing the integers~$1,\ldots,n$ in this order from left to right, from top to bottom. Sometimes we write~$t_{\lambda}$ instead of~$t$. 
It follows} that a representative set for the action of~$S_n$ on~$M^{\mu}$ is given by
\begin{align}\label{Mlambdareprset}
\{(\vartheta_T(e_{t_{\lambda}}) \,\,| \,\, T \text{ semistandard of shape~$\lambda$ and content~$\mu$})  \,\, | \,\, \lambda \vdash n\}.
\end{align}

\paragraph{Induced representations.}
Let~$G$ be a finite group, and~$H$ a subgroup of~$G$. Let~$R=\{r_1,\ldots,r_t\}$ be a full set of representatives for the left cosets of~$H$ in~$G$, so~$|R|=[G\, : \, H]$. If~$V$ is an $H$-module, the \emph{induced} module $\ind_H^G(V)$ is defined as follows. The elements of~$\ind_H^G(V)$ are (formal) sums of the form
$$
\lambda_1 (r_1,v_1) + \ldots + \lambda_t(r_t,v_t)\quad \text{for } v_1,\ldots,v_t \in V, \, \lambda_1,\ldots,\lambda_t \in \C. 
$$
(So as vector space~$\ind_H^G(V) = \oplus_{r \in R} V$.) The action of an element~$g \in G$ on~$(r_i,v)$ is defined via $g \cdot (r_i,v) = (r_j, h\cdot v)$, where~$r_j \in R$ and~$h \in H$ are uniquely determined by the equation~$g r_i = r_j h$.

\subsection{The block-diagonalization for computing~\texorpdfstring{$\alpha_k$}{alphak}}\ytableausetup{centertableaux,boxsize=1.15em}

We aim to decompose the space~$\C^{Z_m}$ as a~$\revise{G_m}$-module. The derivation will consist of three steps. 
\begin{enumerate}
\item Derive a representative set of matrices for the action of~$S_m$ on~$M^{(1^m)}$ from the elementary representation theory of the symmetric group.
\item  There is a natural surjective $G$-homomorphism $f: M^{(1^m)} \to \C^{Z_m}$. For each matrix in the representative set for the action of~$S_m$ on~$M^{(1^m)}$, construct a new matrix consisting of a minimal linearly independent set of columns of the original matrix after applying the map~$f$. The new matrices together form a representative set for the action  of~$S_m$ on~$\C^{Z_m}$, as we will show.

In general: suppose~$G$ is a finite group acting on finite dimensional vector spaces~$V$ and~$W$, and~$f:V\to W$ is a surjective $G$-homomorphism. We show how to derive a representative set for the action of~$G$ on~$W$ from a representative set for the action of~$G$ on~$V$. 
\item Use the additional $S_2 \cong \{\pm 1\}$-action to finally obtain a representative set for the action of~$S_m \times S_2$ on~$\C^{Z_m}$. 

In general: suppose that~$H$ is a finite group acting on a complex finite dimensional vector space~$V$, and that also~$S_2$ acts on~$V$. We show how to derive  a representative set for the action of~$H \times S_2$ on~$V$ from a representative set for the action of~$H$ on~$V$, provided that the $H$- and~$S_2$-actions on~$V$ commute. 
\end{enumerate}
So we first consider the action of the subgroup~$\revise{S_m  \cong S_m\times \{+1\} < S_m \times \{\pm 1\}}$ acting on~$Z_m$ by conjugation, and give an algorithm to determine a representative set for this action. Afterwards, we consider the additional $S_2 \cong \{\pm 1\}$-action to reduce the representative set further.

 %To describe the module~$\C^{Z_m}$ as an~$S_m$-module, we first recall known results about induced representations.

\subsubsection{The \texorpdfstring{$S_m$}{Sm}-action on~\texorpdfstring{$Z_m$}{Zm}}

The starting point to find a representative set for the action of~$S_m$ on~$\C^{Z_m}$ is a representative set for the action of~$S_m$ on~$M^{(1^m)}$ given in~\eqref{Mlambdareprset}.  \iffalse We then use the following identification:
$$
\C^{Z_m} \cong M^{(1^m)}/(\Z/m\Z).
$$
Here~$\Z/m\Z$ acts on a tabloid of shape~$(1^m)$ by permuting its rows cyclically (say, downwards), and a canonical representative of a tabloid~$\{t\} \in M^{(1^m)}$ in $ M^{(1^m)}/(\Z/m\Z)$ is given by permuting the rows cyclically downwards so that the symbol~$1$ appears in the first row. \fi 
%We first obtain a basis for $\text{Hom}(S^{\lambda}, M^{(1^m)}/(\Z/m\Z))$. 
%For this, we consider the \emph{Reynolds operator}~$\rho_{\Z/m\Z}$.  \SP{Why Reynolds operator, and not just the projection operator?}
We consider the natural projection
\begin{equation}
f: M^{(1^m)} \to \C^{Z_m}, \label{projectionp}
\end{equation}
mapping a tabloid which is filled row-wise with $i_1$ up to $i_m$
% $\ytableausetup{smalltableaux, tabloids}\begin{ytableau} i_1\\i_2\\\vdots\\i_m    \end{ytableau}$
to the indicator vector in $\C^{Z_m}$ corresponding to  $(i_1 i_2 \ldots i_m)$.
 % It takes a tabloid~$\{t\}$ in~$M^{(1^m)}$ and the rows are permuted cyclically  so that the symbol~$1$ appears in the first row. This gives a canonical representative of the  equivalence class of~$\{t\}$ in $M^{(1^m)}/(\Z/m\Z)$, and can be naturally identified with an~$m$-cycle from~$Z_m$.
The map~$f$ is linear and surjective, and it respects the~$S_m$-action, as for each~$\pi \in S_m$ we have
\ytableausetup{centertableaux,boxsize=1.6em,tabloids}
$$
f\left( \pi \cdot \begin{ytableau} i_1\\i_2\\\vdots\\i_m    \end{ytableau}\right) = f \left(  \begin{ytableau} \scalebox{0.7}{$\pi(i_1)$}\\\scalebox{0.7}{$\pi(i_2)$}\\\vdots\\\scalebox{0.7}{$\pi(i_m)$}    \end{ytableau} \right) = (\pi(i_1)\ldots \pi(i_m)) =\pi (i_1  \ldots i_m) \pi^{-1} =  \pi f\left(  \begin{ytableau} i_1\\i_2\\\vdots\\i_m    \end{ytableau}\right)\pi^{-1}.
$$\ytableausetup{centertableaux,boxsize=1.15em}We now use the following \revise{fact (which follows from elementary representation theory, see, e.g.,~\cite{Isaacs, Serre})} to derive a representative set for the action of~$S_m$ on~$\C^{Z_m}$.

\begin{proposition}\label{representativeprop}
Suppose that a finite group~$G$ acts on two finite-dimensional complex vector spaces~$V$ and~$W$, and suppose that~$f : V \to W$ is a surjective $G$-homomorphism. Let~$\{U_1,\ldots,U_k\}$ be a representative set for the action of~$G$ on~$V$, with~$U_i = (u_{i,j} \,|\, j=1,\ldots,m_i)$. Then the set $\{U_1', \ldots,U_k'\}$ is representative for the action of~$G$ on~$W$, where~$U_i'$ (for~$i \in [k]$) is a tuple consisting of a minimal spanning set among the $f(u_{i,j})$, with $j=1,\ldots,m_i$.
\end{proposition}
\proof
\revise{For each $i \in [k]$, let~$s_i \in \N$ and $\ell_{1}^{(i)},\ldots,\ell_{s_i}^{(i)} \in [m_i]$ be such that 
$$
U_i'=(f(u_{i,\ell_{1}^{(i)}}), \ldots, f(u_{i,\ell_{s_i}^{(i)}}))
$$
is the chosen tuple consisting of a minimal spanning set among the $f(u_{i,j})$ for $j=1,\ldots,m_i$. Define
$$
V':= \bigoplus_{i=1}^k \bigoplus_{j=1}^{s_i} \C G u_{i,\ell_{j}^{(i)}} \subseteq V,
$$
i.e., $V'$ is the restriction of the direct sum decomposition of~$V$ to the components corresponding to the chosen minimal spanning sets. 

The restriction~$f'\, : \, V' \to W$ of~$f$ to~$V'$ is a bijection.
Surjectivity of~$f'$ is clear, as~$W$, the image of $f$, is spanned by the elements
\[\left\{g\cdot f(u_{i,\ell_{j}^{(i)}}) = f(g\cdot u_{i,\ell_{j}^{(i)}})\mid i \in [k], j \in [s_i], g\in G\right\}.\]
If~$f'$ is not injective, then $\text{Ker}(f')$ contains an irreducible submodule~$M$ of $V'$. By Schur's lemma, the projection of $M$ onto the components $\oplus_{j=1}^{s_i} \C G u_{i,\ell_{j}^{(i)}}$ is zero for all but one~$i \in [k]$. Any nonzero element of~$M$ now gives rise to a nontrivial linear combination of the~$u_{i,\ell_{j}^{(i)}}$  that is in the kernel of~$f$ (for the $i$ for which the projection of~$M$ onto $\oplus_{j=1}^{s_i} \C G u_{i,\ell_{j}^{(i)}}$ is nonzero) contradicting the fact that the~$f(u_{i,\ell_{j}^{(i)}})$ ($j =1,\ldots,s_i$) are linearly independent. So $f'$ is indeed a bijection.

Since by definition the set~$\{(u_{i,\ell_{j}^{(i)}}\, | \, j=1,\ldots, s_i)\,| \, i=1,\ldots, k \} $ is representative for the action of~$G$ on~$V'$, the set
$$
\{U_1',\ldots,U_m'\} = \left\{\left(f'\left(u_{i,\ell_{j}^{(i)}}\right)\, \big| \, j=1,\ldots, s_i\right)\,\big| \, i=1,\ldots, k \right\} 
$$
is representative for the action of~$G$ on~$W$, as was needed to prove.
}
\endproof
Recall that a representative set for the action of~\revise{$S_m$ on~$M^{(1^m)}$} is given by
$$
\{\vartheta_T(e_t) \, | \, T \text{ semistandard of shape~$\lambda$ and content~$(1^m)$}\}.
$$ 
 Note that any semistandard tableaux of shape~$\lambda \vdash m$ and content $(1^m)$ is standard. Consider for each~$\lambda \vdash n$ a tuple~$U_{\lambda}$ consisting of a minimal spanning set among the vectors
\begin{align}\label{vectorsspanning}
\{f (\vartheta_T(e_t)) \, | \, T \text{ standard of shape~$\lambda$ and content~$(1^m)$}\} \subseteq \C^{Z_m}.
\end{align}
\begin{corollary}
The set~$\{U_{\lambda} \, | \, \lambda \vdash n\}$ is representative for the action of~$S_m$ on~$\C^{Z_m}$. 
\end{corollary} 
\proof
Apply Proposition \ref{representativeprop} with~$V=M^{(1^m)}$, $W=\C^{Z_m}$, and $f$ from~\eqref{projectionp}.
\endproof

We note that it is useful to maintain for each~$\lambda$ a list of the Young tableaux which give rise to the minimal spanning set among the vectors in~\eqref{vectorsspanning}. They can help to compute the coefficients in the block-diagonalizations more efficiently (but still exponential in~$m$), see Section~\ref{polynomialinnerproducts}.  

\begin{remark}\label{codingremark}
Proposition \ref{representativeprop} has a wide potential for application.  For instance, for computing bounds on the cardinality of error-correcting codes, a block-diagonalization for matrices indexed by ordered $k$-tuples of codewords can be obtained using existing tools~\cite{gijswijt,polak}. With Proposition~\ref{representativeprop}, one may further reduce this into a block-diagonalization for matrices indexed by \emph{unordered} sets of codewords of size $\leq k$.
\end{remark}

\paragraph{Discussion about finding the minimal spanning set faster.}
It is  also natural to identify~$\C^{Z_m}$ with $M^{(1^m)}/(\Z/m\Z)$, where~$\Z/m\Z$ permutes the rows of a tabloid in~$M^{(1^m)}$ cyclically. Brosch~\cite{brosch} developed a fast method in the context of flag algebras to decompose any module~$M^{\mu}/F$, where~$F$ is a group acting on the rows of~$\mu$ via permutations.  However, the computational results presented in this paper can be obtained without this speed-up: we can compute the representative set for~$\alpha_{k}$ for~$k \leq 10$  using the method from Proposition~\ref{representativeprop}, and the representative set for our new relaxation $\beta_k$ is described explicitly in Section~\ref{sec: betak}.

The method of Brosch~\cite{brosch} allows to avoid working with the vectors $\vartheta_T(e_t)$ explicitly,  which is desirable given the high dimension of $M^{(1^m)}$. The key observation is  
$$\mathrm{Hom}(S^\lambda,M^{(1^m)}/(\Z/m\Z)) = \mathcal{R}_{\Z/m\Z}(\mathrm{Hom}(S^\lambda,M^{(1^m)})),$$
by identifying the quotient $M^{(1^m)}/(\Z/m\Z)$ with the elements $v$ in $M^{(1^m)}$ with $\sigma(v)=v$ for all $\sigma\in \Z/m\Z$. Here $\mathcal{R}_{\Z/m\Z}$ denotes the \emph{Reynolds operator} of $\Z/m\Z$ on $\mathrm{Hom}(S^\lambda,M^{(1^m)})$, which averages over the group
$$\mathcal{R}_{\Z/m\Z} (\vartheta_T) \coloneqq \frac{1}{m}\sum_{\sigma \in \Z/m\Z}\sigma(\vartheta_T).$$
The action of $\Z/m\Z$ on homomorphisms $\vartheta_T$ is given by 
$\sigma(\vartheta_T) = \vartheta_{\sigma(T)},$
where $\sigma$ is applied to $T$ entrywise. The method of \cite{brosch} results in  a matrix representation of $\mathcal{R}_{\Z/m\Z}$ in the semistandard basis, so that one can choose the homomorphisms corresponding to a spanning set of rows to find a basis of $\mathrm{Hom}(S^\lambda,M^{(1^m)}/(\Z/m\Z))$. The advantage is that one works in a space of dimension $\mathrm{dim}(\mathrm{Hom}(S^\lambda, M^{(1^m)}))$ instead of $\mathrm{dim}(M^{(1^m)}) = m!$.

%If we can find a matrix representation of $\mathcal{R}_{\Z/m\Z}$ in the semistandard basis, we can choose the homomorphisms corresponding to a spanning set of rows to find a basis of $\mathrm{Hom}(S^\lambda,M^{(1^m)}/(\Z/m\Z))$ we can describe and use efficiently.

%The problem is that $\sigma(T)$ is generally not semistandard again, and it may not be clear how to express $\vartheta_{\sigma(T)}$ as linear combination of homomorphisms coming from semistandard tableaux. The algorithm of \cite{brosch} computes this decomposition. If we now use it to determine the matrix representing the operation of a generator of $\Z/m\Z$, we can use it to compute the Reynolds operator explicitly, and thus decompose $M^{(1^m)}/(\Z/m\Z)$ into irreducible Specht modules.

As mentioned before, knowing the description of the columns $\vartheta_T(e_t)$ of the representative set  in terms of tableaux is useful for the computations, see Section~\ref{polynomialinnerproducts}.

%: We can apply the algorithms in Sven's thesis to compute the symmetrized products efficiently. This is crucial to make the computation time reasonable.
%Daniels Algorithm:
%Hom(Smu, Mlambda/G) = R_G(Hom(Smu,M^lambda))

%we start with semistandard tableaux, then for each one we look how each group element how it acts, gives matrix and then average. 

%Reynolds operator sends Mlambda to quotient 

\paragraph{The multiplicities of the irreducible representations.}
It can be shown that the module~$\C^{Z_m}$ is $S_m$-isomorphic to a module which has been described in the literature. This allows us to immediately obtain the  multiplicities of the irreducible representations of~$\C^{Z_m}$ as an~$S_m$-module.
\begin{proposition}
As~$S_m$-modules, we have~$\C^{Z_m} \cong \ind_{\Z/m \Z}^{S_m} 1$.
\end{proposition}
\proof 
Define the map $\phi:\C^{Z_m} \to \ind_{\Z/m \Z}^{S_m} 1$ by mapping the standard basis vector $e_\sigma$ corresponding to~$\sigma=(\sigma_1 \,\sigma_2 \ldots \sigma_m) \in Z_m$ with $\sigma_1=1$ to the basis element $(r,1)$ in $\ind_{\Z/m \Z}^{S_m} 1$, where~$r$ is the permutation which maps~$i$ to~$\sigma_i$ for each~$i \in [m]$. Then
\begin{align}
    \phi(\pi\cdot e_{\sigma}) = \phi (e_{\pi \sigma \pi^{-1}}) = \phi (e_{(\pi \sigma_1 \,\pi \sigma_2 \,\ldots \,\pi \sigma_m)}) = (\overline{\pi r},1) = \pi \cdot \phi(e_{\sigma}),
    \end{align}
    for each~$\pi \in S_m$, where $\overline{\pi r}$ is the representative of the class of the permutation $\pi r$ with $\overline{\pi r}(1) = 1$. 
So~$\phi$ respects the~$S_m$-action.  As~$\phi$ is also a bijection between the bases of~$\C^{Z_m}$ and $\ind_{\Z/m \Z}^{S_m} 1$, its linear extension is an~$S_m$-isomorphism.
\endproof 
\noindent It is known~\cite{liepaper} that 
\begin{align}\label{eq:inducedalambda}
    \text{Ind}_{\Z/m \Z}^{S_m} 1 \cong \bigoplus_{\lambda \vdash m} a_{\lambda} S^{\lambda},
\end{align}
where $a_{\lambda}$ is the number of standard tableaux~$T$ of shape~$\lambda$ with $c(T)=0 \pmod{m}$, where~
\begin{align}
&c(T) \text{ is }\text{the sum of all $a$ in~$T$ for which~$a+1$ appears in a row} \notag\\&\text{strictly below $a$'s row}.
\end{align}
  So it is not hard to determine the multiplicities of the irreducible representations of $\C^{Z_m}$ as $S_m$-module. We however need the decomposition explicitly, to obtain an explicit representative set.

\subsubsection{The \texorpdfstring{$S_2\cong \{\pm 1\}$}{S2}-action on~\texorpdfstring{$Z_m$}{Zm}}
The~$S_m$-action and the~$S_2 \cong \{ \pm 1 \}$-action on~$\C^{Z_m}$ commute. This enables us to compute a representative set for the action of~$S_m \times S_2$ on~$\C^{Z_m}$, starting with a given representative set for the action of~$S_m$ on~$Z_m$. We first state the setting in a general form, and then prove a proposition which allows us to derive the full symmetry reduction. 

%Actions of~$S_2$ are common and have many applications. For example, in~\cite[Section 3.4]{quadruples}  a representative set of  the action of~$S_2$ on a vector space~$\C^Z$ is given, where~$Z$ is a finite set. Here we show how to directly derive a representative set for the action of~$H \times S_2$ on~$V$, from a representative set for the action of $H$ on~$V$ (if~$H$ and~$S_2$ both act on  a finite dimensional complex vector space~$V$ and these actions commute). 

%So, we can compute the block-diagonalization corresponding to the~$S_m$-action, and subsequently compute the block-diagonalization corresponding to the $\{\pm 1\}$-action. 

\subsubsection{Representative set of~\texorpdfstring{$H \times S_2$}{HxS2}-action}
Let~$H$ be a finite group acting on a finite-dimensional complex vector space $V$ and suppose a representative set~$\{U_1,\ldots,U_k\}$ where~$U_i=(u_{i,1},\ldots,u_{i,m_i})$ (for~$i \leq k$) for the action of~$H$ on $V$ is given. Suppose that also~$S_2=\{1,\eta\}$ acts on~$V$, and that the actions of~$H$ and~$S_2$ on~$V$ commute. Let~$L_{\pm} := \{ x \, | x= \pm \eta x\}$, so that $L_+$ and $L_-$ are the eigenspaces of~$\eta$. Proposition~\ref{s2proposition} shows how to obtain a representative set for the action of~$H \times S_2$ on~$V$, generalizing~\cite[Section 3.4]{quadruples} (which considers~$S_2$-actions on a finite set~$Z$).
\begin{proposition}\label{s2proposition}
A representative set for the action of~$H\times S_2$ on $V$ is the set
$\{U_1^+, U_1^-, \ldots, U_k^+, U_k^-\}$, where $U_i^+$ is a tuple consisting of a linearly independent subset among the vectors $u_{i,j}^+:= u_{i,j} + \eta \cdot u_{i,j}$ (for $j=1,\ldots,m_i$), and $U_i^-$ is a tuple consisting of a linearly independent subset among the vectors $u_{i,j}^-:= u_{i,j} - \eta \cdot u_{i,j}$ (for $j=1,\ldots,m_i$).
\end{proposition}
\proof 
\revise{Since the actions of~$H$ and~$S_2$ on~$V$ commute, both~$L_+$ and~$L_-$ are $H \times S_2$-invariant subspaces of~$V$. The maps~$f^+ : V \to L_+$ and $f^- : V \to L_-$ given by~$f^+(v)=(I+\eta)v$ and $f^-(v)=(I-\eta)v$ are surjective~$H \times S_2$-homomorphisms. From Proposition~\ref{representativeprop} it now follows that~$\{U_1^+,\ldots,U_k^+\}$ and~$\{U_1^-,\ldots,U_k^-\}$ are representative sets for the actions of~$H\times S_2$ on~$L_+$ and~$L_-$, respectively. 

Note that~$V=L_+ \oplus L_-$. Also, if~$W_1 \subseteq L_+$ and~$W_2 \subseteq L_-$, are irreducible $H\times S_2$-modules, then they are non-isomorphic: indeed, if~$\psi:W_1\to W_2$ were an $H\times S_2$-isomorphism, then for each~$x\in W_1$ we have~$\psi(x)=\psi(\eta x)=\eta\psi(x)$, as~$x \in L_+$, but also~$\psi(x)=-\eta\psi(x)$, as~$\psi(x)\in L_-$, so~$\psi(x)=0$. So the union $\{U_1^+,\ldots,U_k^+\} \cup \{U_1^-,\ldots,U_k^-\}$ of representative sets for the actions of~$H \times S_2$ on~$L_+$ and~$L_-$ is a representative set for the action of~$H\times S_2$ on~$V$. }
\endproof

For our semidefinite program this means that, in the block-diagonalization for the action of~$S_m$ on $\C^{Z_m}$, the block corresponding to the matrix~$U_\lambda$ will split into two blocks in the block-diagonalization for the action of~$S_m \times S_2$ on~$\C^{Z_m}$: one corresponding to~$U_{\lambda}^+$ and one corresponding to~$ U_{\lambda}^-$.

\section{The relaxation~\texorpdfstring{$\beta_m$}{betak}\label{sec: betak}}
When computing~$\alpha_m$, we use the symmetry reduction from the previous section and require that all blocks in the block-diagonalization of~$X$ are positive semidefinite. As~$\alpha_m$ is a minimization problem, only requiring one block to be positive semidefinite will yield a lower bound on~$\alpha_m$. From our computer experiments it follows that one small block seems `special': only requiring this block to be positive semidefinite yields a remarkably good lower bound on~$\alpha_m$. It is the block corresponding to~$U_{\lambda}^-$, where~$\lambda=(m-2,1,1) \vdash m$.  This observation gives rise to a new relaxation~$\beta_m$ of~$\alpha_m$, in which we only require the mentioned block to be positive semidefinite. 
The primal of the program~$\beta_m$ is
\begin{align} \label{betam:primal}
    \beta_m = \min \left\{\langle Q, X\rangle \, | \, X \in \R^{Z_m \times Z_m}_{\geq 0}, \, \langle J,X \rangle =1, \, (U_{\lambda}^{-})\T X U_{\lambda}^{-}  \succeq 0  \right\},  
\end{align}
where~$\lambda=(m-2,1,1)$. It turns out that we can explicitly describe the columns of the matrix $U_{\lambda}^{-}$ using Young tableaux. We first describe the matrix~$U_{\lambda}$. Define the tableau
$$\ytableausetup{notabloids}
M_i :=  \ytableaushort{{}{}{\cdots}{}{},{2},{i}}, \quad \text{ for $i \in \{3,\ldots,m\}$}. 
$$
%We claim that the tableaux~$M_3,\ldots, M_{\floor{\tfrac{m+1}{2}}+1}$ or computing~$\beta_m$. 

%\subsection{Proof of the correct representative set}

\begin{proposition}\label{betacombinatorialproposition}
The matrix~$U_{\lambda}$ can be chosen to be the matrix consisting of the columns~$f(\vartheta_{M_i}(e_t))$ for $i=3,\ldots, \floor{\tfrac{m+1}{2}}+1$. 
\end{proposition}
\proof
First, we calculate $a_{\lambda}$ from~\eqref{eq:inducedalambda} for the partition~$\lambda=(m-2,1,1) \vdash m$. Recall that~$a_{\lambda}$ is the number of semistandard tableaux~$T$ with~$c(T)= 0 \pmod{m}$. Suppose that a standard tableau~$T$ has~$a$ and~$b$ as entry in its second and third row, so~$1<a<b\leq m$. Moreover~$c(T)$ is zero modulo~$m$ if and only if~$(a-1)+(b-1)=0 \pmod{m}$. There are exactly $\floor{\tfrac{m-1}{2}}$ pairs~$a,b$ satisfying $1<a<b\leq m $ and $a+b=m+2$, so $a_{\lambda} = \floor{\tfrac{m-1}{2}}$. So the number of columns of the matrix~$U_{\lambda}$ is $ \floor{\tfrac{m-1}{2}}$, which is exactly the number of vectors~$f(\vartheta_{M_i}(e_t))$ given in this proposition.

We now show that if~$T$ is any standard tableau of shape~$\lambda$, then $  f(\vartheta_{T}(e_t)) = f(\vartheta_{M_i}(e_t))$ for some~$i= 3,\ldots, \floor{\tfrac{m+1}{2}}+1$. It then follows that the given set of columns is a spanning set for the column space of~$U_{\lambda}$, and by the previous paragraph it has the correct size, so it is minimal and we are done. Note that if
$$
T_1 = \ytableaushort{{}{}{\cdots}{}{},{a_1},{b_1}} \text{ and } T_2=\ytableaushort{{}{}{\cdots}{}{},{a_2},{b_2}}
$$
are standard of shape $(\revise{m-2},1,1)$ and content~$(1^m)$, with $b_1-a_1 = b_2-a_2$, then $f(\vartheta_{{T_1}}(e_t)) = f(\vartheta_{T_2}(e_t))$. To see this, note that 
\begin{align}\label{betavectorexplicit}\ytableausetup{centertableaux,boxsize=1.4em,tabloids}
\vartheta_{{T_1}}(e_t) = \sum \hspace{1pt}\begin{ytableau}     
 \raisebox{1pt}{\scalebox{0.65}{$\substack{\ast\\\ast\\\ast}$}} \\
 \scalebox{0.75}{$m\hspace{-3pt}-\hspace{-3pt}1$}  \\
 \raisebox{1pt}{\scalebox{0.65}{$\substack{\ast\\\ast\\\ast}$}} \\
   m  \\
   \raisebox{1pt}{\scalebox{0.65}{$\substack{\ast\\\ast\\\ast}$}}\\
\end{ytableau} - \sum \hspace{1pt}\begin{ytableau}     
  \raisebox{1pt}{\scalebox{0.65}{$\substack{\ast\\\ast\\\ast}$}} \\
  m  \\
  \raisebox{1pt}{\scalebox{0.65}{$\substack{\ast\\\ast\\\ast}$}} \\
  \scalebox{0.75}{$m\hspace{-3pt}-\hspace{-3pt}1$}  \\
  \raisebox{1pt}{\scalebox{0.65}{$\substack{\ast\\\ast\\\ast}$}} \\
\end{ytableau} -\sum\hspace{1pt} \begin{ytableau}     
    \raisebox{1pt}{\scalebox{0.65}{$\substack{\ast\\\ast\\\ast}$}} \\
 1  \\
  \raisebox{1pt}{\scalebox{0.65}{$\substack{\ast\\\ast\\\ast}$}} \\
 m  \\
  \raisebox{1pt}{\scalebox{0.65}{$\substack{\ast\\\ast\\\ast}$}} \\
\end{ytableau}+\sum\hspace{1pt} \begin{ytableau}     
  \raisebox{1pt}{\scalebox{0.65}{$\substack{\ast\\\ast\\\ast}$}} \\
  m \\
   \raisebox{1pt}{\scalebox{0.65}{$\substack{\ast\\\ast\\\ast}$}} \\
 1  \\
  \raisebox{1pt}{\scalebox{0.65}{$\substack{\ast\\\ast\\\ast}$}} \\
\end{ytableau}-\sum\hspace{1pt} \begin{ytableau}     
  \raisebox{1pt}{\scalebox{0.65}{$\substack{\ast\\\ast\\\ast}$}} \\
  \scalebox{0.75}{$m\hspace{-3pt}-\hspace{-3pt}1$}  \\
  \raisebox{1pt}{\scalebox{0.65}{$\substack{\ast\\\ast\\\ast}$}} \\
 1  \\
  \raisebox{1pt}{\scalebox{0.65}{$\substack{\ast\\\ast\\\ast}$}} \\
\end{ytableau}+\sum\hspace{1pt} 
\begin{ytableau}     
    \raisebox{1pt}{\scalebox{0.65}{$\substack{\ast\\\ast\\\ast}$}} & \none[]\\
 1 & \none[ \hspace{1.5cm}\leftarrow \text{row } a_1]   \\
    \raisebox{1pt}{\scalebox{0.65}{$\substack{\ast\\\ast\\\ast}$}} & \none[ ] \\
 \scalebox{0.75}{$m\hspace{-3pt}-\hspace{-3pt}1$} & \none[ \hspace{1.5cm}\leftarrow \text{row } b_1]  \\
    \raisebox{1pt}{\scalebox{0.65}{$\substack{\ast\\\ast\\\ast}$}} & \none[ ]\\
\end{ytableau}\hspace{-12pt},\hspace{30pt}
\end{align}
where each sum is over all \revise{tabloids of shape and content~$(1^m)$ with the given fixed entries in rows $a_1$ and $b_1$. Thus, each sum is over $(m-2)!$ tabloids.}
The vector~$\vartheta_{T_2}(e_t)$ is obtained from~\eqref{betavectorexplicit} upon replacing~$a_1$ and~$b_1$ by~$a_2$ and~$b_2$, respectively.
As~$b_1-a_1=b_2-a_2$, each term in the sum expansions of $\vartheta_{{T_1}}(e_t)$ and $\vartheta_{{T_2}}(e_t)$ represent, after projection, the same element of~$Z_m$. So  $f(\vartheta_{{T_1}}(e_t)) = f(\vartheta_{T_2}(e_t))$. 
So the vector~$f(\vartheta_{T_1}(e_t))$ is the same as one of the $f(\vartheta_{M_i}(e_t))$ with~$ 3 \leq i \leq m$, namely the one with $i-2=b-a$. 

The proof is completed by observing that  $f(\vartheta_{M_{m-i}}(e_t))=  f(\vartheta_{M_{i+4}}(e_t))$ for all~$i=0,\ldots,m-4$, as the projection of any~$\vartheta_{M_j}(e_t)$ onto~$Z_m$ only depends on the distance between~$j$ and~$2$ mod $m$. The distinct nonzero distances mod~$m$ between~$i$ and~$2$ are~$1,\ldots,\floor{\tfrac{m-1}{2}}$, which corresponds to~$i= 3,\ldots, \floor{\tfrac{m+1}{2}}+1$.  
So  if~$T$ is any standard tableau of shape~$\lambda$, then~$  f(\vartheta_{T}(e_t)) = f(\vartheta_{M_i}(e_t))$ for some~$i= 3,\ldots, \floor{\tfrac{m+1}{2}}+1$.
\endproof

It is not hard to verify using~\eqref{betavectorexplicit} that $\eta \cdot f(\vartheta_{T_i}(e_t)) = -f(\vartheta_{T_i}(e_t))$, where~$\eta$ is the inversion action on~$Z_m$. From this it follows that the columns of $U_{\lambda}^{-}$ can be taken to be the same columns as the columns of~$U_{\lambda}$, and that the matrix~$U_{\lambda}^{+}$ is the zero matrix. In Section~$5$ we will therefore only work with the matrix~$U_{\lambda}$ and not with the matrix~$U_{\lambda}^-$.

\section{Computation \label{sec: computation}}

In this section, we comment on the computation. First we explain how we compute the entries of~$Q$, taking into account its symmetries. After that, we describe how to compute the entries in the block-diagonalizations more efficiently. 
Then we give the dual semidefinite program of~$\beta_m$, which has nice features: a small matrix block which is required to be positive semidefinite, and few variables. However, it has $|\Omega_m'|$  linear constraints, which is a very large number.\footnote{Recall that~$\Omega_m:=  (Z_m \times Z_m)/G_m$ is the collection of nonempty $G_m$-orbits of~$Z_m \times Z_m$, and~$\Omega_m'$ is the collection of nonempty $G_m$-orbits on~$Z_m \times Z_m$ in which additionally orbits of $(\sigma,\tau)\in Z_m\times Z_m$ and~$(\tau,\sigma)$ are identified.} In the final section we explain how we computed~$\beta_m$ using this dual description in practice. 

\subsection{Computing the matrix~\texorpdfstring{$Q$}{Q} with Dijkstra's algorithm}
To compute the entries of the matrix~$Q$, we follow Woodall~\cite{woodall}. Construct a graph~$\Gamma_m$ with vertex set~$Z_m$, and~$\{\sigma,\gamma\}$ is an edge if~$\gamma$ can be obtained from~$\sigma$ by one transposition of adjacent elements of~$\sigma$. Then the entry~$Q_{\sigma,\tau}$ is equal to the length of a shortest path from~$\sigma$ to~$\tau^{-1}$ in~$\Gamma_m$, which can be computed with Dijkstra's shortest path algorithm. 
We only apply Dijkstra with the source node~$\sigma=(12\ldots m)$,   as we only want the value of~$Q_{\sigma,\tau}$ on $G_m$-orbits of~$Z_m \times Z_m$.  

A speed-up inside Dijkstra algorithm which takes into account symmetry is based on the observation that~$\sigma = (12\ldots m)$ is fixed by the elements~$(\sigma,1)$ and~$(\rho,-1)$ of~$G_m$, where $\rho$ is such that~$\rho \sigma^{-1} \rho^{-1} = \sigma$. So the subgroup $H_m$ of~$G_m$ generated by these two elements fixes~$\sigma$, and hence has the property that $Q_{\sigma, h \cdot \tau} = Q_{h \cdot \sigma, h \cdot \tau} = Q_{\sigma, \tau}$ for any $h \in H_m$ and~$\tau \in Z_m$.  We represent each~$H_m$-orbit of~$Z_m$  by its lexicographically smallest element.  We maintain a priority queue~$S$ of elements with their distances, and a set~$L$ of visited orbit representatives of~$Z_m$ under~$H_m$, and a distance $d:=0$. The priority queue~$S$ initially consists of~$(12\ldots m)$ with distance $0$, and~$L$ consists of~$\sigma=(12\ldots m)$.

\revise{As long as there are orbits in $S$, we pop the element~$\tau$ from~$S$ with the smallest distance, increase~$d$ by~$1$, and check all cycles in~$ Z_m$ reachable from $\tau$ with one swap of adjacent elements in~$\tau$. These cycles are replaced with the unique representatives of their orbits, and the new orbit representatives are added to $L$, as well as to the queue~$S$ with distance~$d$. This is repeated until $S$ is empty.} %We repeat this paragraph until $S$ is empty. 

\subsection{Computing the inner products\label{polynomialinnerproducts}}

Let~$\lambda \vdash m$ and~$u_{T_1}=f(\vartheta_{T_1}(e_{t_{\lambda}}))$, $u_{T_2}=f(\vartheta_{T_2}(e_{t_{\lambda}}))$ be columns of~$U_{\lambda}$. Let $X \in (\C^{Z_m \times Z_m})^{G_m}$. The inner products are of the form %= 2u_{T_1}\T X u_{T_2}+\eta(2u_{T_1}\T X u_{T_2}),
$$
((1 + \eta) \cdot u_{T_1})\T X ((1 +  \eta)\cdot  u_{T_2}) \,\,\,\,\, \text{ or }\,\,\,\,\, ((1 - \eta)\cdot  u_{T_1})\T X ((1 -  \eta)\cdot  u_{T_2}).
$$
By symmetry one has $(\eta \cdot u_{T_1})\T X  (\eta\cdot   u_{T_2}) = u_{T_1}\T X   u_{T_2}$ and $(\eta\cdot  u_{T_1})\T X   u_{T_2} =  u_{T_1}\T X  (\eta \cdot   u_{T_2})$. So to compute the inner products, we must compute expressions of the form $ u_{T_1}\T X   u_{T_2}$ and $(\eta \cdot  u_{T_1})\T X   u_{T_2}$.  Note that
\begin{align} \label{eq:innerprod}
u_{T_1}\T X u_{T_2} =\sum_{\substack{T_1' \sim T_1,  T_2' \sim T_2}} \sum_{c,c' \in C_{t}}  \text{sgn}(cc') x_{\omega(f(t[c T_1']), f(t[c'  T_2']))},
\end{align} 
where~$f$ from~\eqref{projectionp} maps a tabloid to the corresponding~$m$-cycle in~$Z_m$, and $\omega(\sigma,\tau) \in \Omega_m'$ denotes the orbit of~$(\sigma,\tau)\in Z_m \times Z_m$. If we have~$\eqref{eq:innerprod}$, then one can also obtain $(\eta \cdot  u_{T_1})\T X   u_{T_2}$ from it by replacing each variable $x_{\omega(f(t[c T_1']), f(t[c'  T_2']))}$ by $x_{\omega(\eta \cdot f(t[c T_1']), f(t[c'  T_2']))}$. So we now focus on computing~$\eqref{eq:innerprod}$. One can compute the inner products by using~\eqref{eq:innerprod} (and we succeeded to compute~$\alpha_{10}$ in that way). We now describe a method which is faster in practice and which we used in our implementation. Since $|\Omega_m'|$ is exponential in~$m$, one cannot hope for a running time polynomial in~$m$.   Let~$Y(\lambda)$ be the set of (row,column)-coordinates indicating the boxes of~$\lambda$. Define the polynomial
\begin{align}\label{sympoly}
p_{T_1,T_2}(Z):=\sum_{\substack{T_1' \sim T_1, T_2' \sim T_2}} \sum_{c,c' \in C_{t}} \text{sign}(cc')\prod_{y \in Y(\lambda)} z_{c T_1'(y),c'T_2'(y)},
\end{align}
for $Z=(z_{j,h})_{j,h=1}^m \in \R^{m \times m}$. 
One can express~$p_{T_1,T_2}$ as a linear combination of monomials with the algorithms of \cite{gijswijt} or~\cite{LPS}. This allows to compute the inner product fast in many instances for error-correcting codes (see e.g.,~\cite{gijswijt, polak}).  The method was generalized to be applicable to arbitrary permutation modules in the setting of flag algebras (cf.~\cite{brosch}). 

There is a one-to-one correspondence between $S_m$-orbits of pairs of tabloids $(t[cT_1], t[c'  T_2])$ and monomials $ \prod_{y \in Y(\lambda)} z_{cT_1'(y),c'T_2'(y)}$ via their  \emph{overlap}, i.e., the numbers of elements of each row of the first tabloid which appear in each row of the second. The overlap of two tabloids~$\{t_1\}$ and~$\{t_2\}$ can be described by a monomial
$
\prod_{i,j=1}^m z_{i,j}^{ (| \{t_1\}_i \cap \{t_2\}_j|)},
$
where~$m$ is the number of parts of~$\lambda$ and~$\{t\}_i$ denotes the set of elements in the $i$-th row of a tabloid~$\{t\}$. So to compute~\eqref{eq:innerprod}, we can compute~\eqref{sympoly}, and then replace each monomial of degree~$m$ in the variables~$z_{i,j}$ by  the  variable $x_{\omega(t[cT_1], t[c'  T_2])}$, where $(t[cT_1], t[c'  T_2])$ is any element in the $S_m$-orbit of pairs of tabloids corresponding to the monomial in~$z_{i,j}$.

\paragraph{Computing~\eqref{sympoly}.} We here state the method from~\cite{gijswijt}, which is easy to implement and uses only methods for addition, multiplication, and differentiation of polynomials. Given two generalized Young tableaux~$T_1,T_2$, define 
\begin{alignat*}{3} 
&r(s,j) := \text{number of $s$'s in row $j$ of $T_1$}, \quad\,\, &&u(s,j) := \text{number of $s$'s in row $j$ of $T_2$},\notag\\ 
&d_{s \to j}:= \sum_{i=1}^m x_{s,i} \frac{\partial}{\partial x_{j,i}}, \,\, \text{ and } && d_{j \to s}^*:= \sum_{i=1}^m x_{i,s} \frac{\partial}{\partial x_{i,j}}.\notag
\end{alignat*} 
Also, define the polynomial 
$
P_{\lambda}(Z):= \prod_{k=1}^m \left( k! \, \text{det} \left(  (z_{i,j})_{i,j=1}^{k} \right) \right)^{\lambda_{k}-\lambda_{k+1}}$ in variables~$z_{i,j}$, where~$i,j \in [m]$ and   $\lambda_{m+1}:=0$.
Then it holds~\cite[Theorem~7]{gijswijt} that
\begin{align*}  
    p_{T_1, T_2 }(X) = \left( \prod_{j=1}^{m-1} \prod_{s=j+1}^m   \frac{1}{r(s,j)!\,u(s,j)!}    (d_{s \to j})^{r(s,j)}  (d_{j \to s}^*)^{u(s,j)}\right) \cdot  P_{\lambda}(Z).
\end{align*}\vspace{-17pt}

\subsection{The dual semidefinite program}

First, note that the dual of the original semidefinite program~$\alpha_m$ is
\begin{align}\label{alpha:dual}\alpha_m= \max\{t \, | \, Q - tJ - Y \succeq 0, Y\in \R^{Z_m \times Z_m}_{\geq 0} \}.\end{align}
To show that this is indeed an equality, one needs to show that strong duality holds. This is indeed the case, as the primal~\eqref{mainsdp} is strictly feasible (set $X=aJ + bI$, where 
$a=\tfrac{1}{2((m-1)!)^2}$ and $b=\tfrac{1}{2(m-1)!}$), while the dual is feasible with $t=0$ and~$Y=Q-\Delta(Q)$, where~$\Delta(Q)$ is a matrix which is zero outside the diagonal and which has the same diagonal entries as~$Q$.

\begin{table}[ht]\small
    \centering
    \begin{tabularx}{1.0\textwidth}{rrrXr}\toprule
$m$&  $|\Omega_m| $ & $|\Omega_m'|$ & \text{block sizes $m_i$ for $\alpha_m$} & $\sum m_i$\\\midrule
         4  & 3 &3 & $1^{3}$ & 3\\
         5  & 8 &  7 & $2^{1}1^{4}$ & 6\\
         6  & 20  & 17& $2^{3}1^{8}$ & 14\\
         7  & 78 & 56 & $ 3^{6}2^{4}1^{8}$ & 34\\
         8  & 380  & 239 &$7^{2}5^{2}4^{9}3^{7}2^{4}1^{9}$ & 98 \\
         9  & 2438 & 1366 & $12^{8}11^{2}9^{6}7^{3}6^{5}5^{2}4^{2}3^{16}1^{5}$ & 294 \\[2pt]
         10 & 18744  & 9848 &  $38^{2}34^{1}31^{1}29^{1}28^{1}26^{3}24^{2}22^{4}20^{5}18^{3}16^{4}14^{6}13^{1}12^{2}10^{4}\allowbreak 9^{1}8^{7}6^{8}4^{7}3^{1}2^{7}1^{3}$ & 952\\[2pt]
         11 &166870 & 85058 & $105^{4}80^{2}60^{6}56^{4}55^{2}54^{2}50^{8}45^{6}44^{2}40^{6}\allowbreak34^{2}30^{6}29^{2}26^{2} 25^{2} 24^{2}20^{6}16^{2}15^{2}11^{4}10^{8}6^{4}\allowbreak5^{14}4^{2}1^{2}$ & 3246 \\[12pt]
         12 & 1670114 &840906  &  $327^{1}317^{1}243^{1}241^{2}238^{1}234^{4}199^{1}191^{1}187^{1}177^{1}176^{1}172^{1}169^{1}163^{1}162^{2}155^{2}\allowbreak150^{1}147^{1}146^{4}144^{1}137^{2}133^{2}  132^{1}128^{1}127^{1}121^{1}117^{3} 113^{1}110^{1}106^{1}102^{1}98^{3}93^{2}\allowbreak 91^{1}90^{2}87^{4} 86^{2}84^{1}83^{1}82^{2}81^{1}79^{1}77^{1}75^{1}74^{1}72^{4}71^{2}68^{1} 66^{1}64^{2}59^{4}56^{2}50^{1}49^{1}47^{2}\allowbreak45^{1}44^{3}41^{2}37^{3}36^{1}34^{1}32^{2}29^{1} 26^{1}\allowbreak 25^{1}  24^{1}19^{3}17^{3}16^{1}14^{5}13^{4}12^{3}10^{1}9^{4}7^{5} 6^{1}5^{3}4^{1}3^{1}\allowbreak2^{5}1^{2}$  & 11698 \\       
         13 & 18446184 & 9244958 & & \\\bottomrule     
    \end{tabularx}
    \caption{%The column~`$S_m$-orbits' denotes the number $|(Z_m \times Z_m)/S_m|$, i.e., the number of orbits without exploiting the additional $\{\pm 1\}$-action. 
 The number of variables in our SDP is $|\Omega_m'| = \sum m_i(m_i+1)/2$, and for the block sizes~$m_i$ for computing~$\alpha_m$ we have~$\sum m_i^2 = |\Omega_m|=|(Z_m \times Z_m)/G_m|$. The block sizes are given in the format $(\text{block size})^{\text{multiplicity}}$.\label{tableorbitslabel}}
\end{table}

We now describe the dual of~$\beta_m$. The primal of the program~$\beta_m$ is
\begin{align}\label{betaprimalwithoutsymmetry}
    \beta_m = \min \left\{\langle Q, X\rangle \, | \, X \in \R^{Z_m \times Z_m}_{\geq 0}, \, \langle J,X \rangle =1, \, U_{\lambda}\T X U_{\lambda}  \succeq 0  \right\},  
\end{align}
where~$\lambda=(m-2,1,1)$.  For each~$\omega \in \Omega_m'$, let~$K_{\omega}$ be the indicator matrix of~$\omega$, i.e., the~$(Z_m \times Z_m)$-matrix with $(K_{\omega})_{\sigma,\tau}=1$ if $(\sigma,\tau) \in \omega$ and $(K_{\omega})_{\sigma,\tau}=0$ otherwise. %\revise{The matrix $K_{\omega}$ is a sum of up to two \emph{basis matrices} of orbitals, cf.\ \cite{Cameron1999}.}\revise{(This is known as the (symmetrized) \emph{basis matrix} of the orbital $\omega$, cf.\ \cite{Cameron1999}.)}
%$$
%(K_{\omega})_{\sigma,\tau} := \begin{cases} 1   &  \text{if $(\sigma,\tau) \in \omega$,}  \\ 0  & \text{otherwise}. \end{cases}
%$$
As~$X$ is~$G_m$-invariant, we may write~$X= \sum_{\omega \in \Omega_m'} K_{\omega} x_{\omega}$. We define for each~$\omega \in \Omega_m'$ the constant matrix~$A_{\omega}:= U_{\lambda}\T K_{\omega} U_{\lambda}$. Let~$q_{\omega}$ denote the common value of~$Q_{(\sigma,\tau)}$ for $(\sigma,\tau) \in \omega$. So we may rewrite~\eqref{betaprimalwithoutsymmetry} as
\begin{align}
    \beta_m &= \min \big\{\sum_{\omega \in \Omega_m'} |\omega| x_{\omega}q_{\omega} \, : \, x_{\omega} \geq 0  \, \forall \omega \in \Omega_m', \,\sum_{\omega \in \Omega_m'}  |\omega| x_{\omega}= 1, \,  \sum_{\omega \in \Omega_m'} x_{\omega} A_{\omega} \succeq 0  \big\}.\notag
\end{align}
The dual of this semidefinite program is (again strong duality holds)
\begin{align}\label{betamdual}
    \beta_m = \max \big\{ t \, :  \, Y \in \R^{\floor{\tfrac{m-1}{2}} \times \floor{\tfrac{m-1}{2}}}, \,  Y \succeq 0, \,  \forall \omega \in \Omega_m'\, : \, \langle Y, A_{\omega}\rangle + |\omega| t \leq |\omega|q_{\omega} \big\}. 
\end{align}
This dual has few variables and only a very small matrix block which is required to be positive semidefinite. The main difficulty is that there are many linear constraints, as can be seen in Table \ref{tableorbitslabel}.

\begin{remark}
We observed some structure in the optimal solutions $Y$ of the dual \eqref{betamdual} of $\beta_m$ computationally. Up to $m=13$, the rank of the optimal $Y$ is one if $m$ is odd, and $2$ if $m$ is even (and $m>4$). Furthermore, the eigenvector of the cases where $m$ is odd behaves similarly for each $m$, as can be seen in Figure \ref{figure:betaSolutions}. This gives us some hope that the optimal solutions can be constructed analytically, potentially leading to improved bounds for bigger $m$ in the future.

\begin{figure}[ht]
\centering %\scalebox{1}{
\begin{tikzpicture}
\begin{axis}[width={0.7\textwidth}, height={0.5\textwidth}, tick label style={font=\footnotesize},xmin={0.85}, xmax={6.15}, xtick={{1.0,2.0,3.0,4.0,5.0,6.0}}, xticklabels={{$M_3$,$M_4$,$M_5$,$M_6$,$M_7$,$M_8$}}, ymin={0}, ymax={2.5}, ytick={{0.0,0.5,1.0,1.5,2.0,2.5}}, yticklabels={{$0.0$,$0.5$,$1.0$,$1.5$,$2.0$,$2.5$}},axis x line*=bottom,axis y line*=left]
    \addplot[color={donkerblauw}, draw opacity={1.0}, line width={1}]
        table[row sep={\\}]
        {
            \\
            1.0  0.5477225575051661  \\
            2.0  0.3385111569432115  \\
        }
        ;
    \addplot[color={donkerblauw}, only marks, draw opacity={1.0}, line width={0}, solid, mark={*}, mark size={2.0 pt}, mark repeat={1}, mark options={color={rgb,1:red,0.0;green,0.0;blue,0.0}, draw opacity={1.0}, fill={donkerblauw}, fill opacity={1.0}, line width={0.75}, rotate={0}, solid}]
        table[row sep={\\}]
        {
            \\
            1.0  0.5477225575051661  \\
            2.0  0.3385111569432115  \\
        }
        ;
    \addplot[color={donkerblauw}, draw opacity={1.0}, line width={1}, solid]
        table[row sep={\\}]
        {
            \\
            1.0  0.9241589976025947  \\
            2.0  0.7763005370264053  \\
            3.0  0.46693002673905953  \\
        }
        ;
    \addplot[color={donkerblauw}, only marks, draw opacity={1.0}, line width={0}, solid, mark={*}, mark size={2.0 pt}, mark repeat={1}, mark options={color={rgb,1:red,0.0;green,0.0;blue,0.0}, draw opacity={1.0}, fill={donkerblauw}, fill opacity={1.0}, line width={0.75}, rotate={0}, solid}]
        table[row sep={\\}]
        {
            \\
            1.0  0.9241589976025947  \\
            2.0  0.7763005370264053  \\
            3.0  0.46693002673905953  \\
        }
        ;
    \addplot[color={donkerblauw}, draw opacity={1.0}, line width={1}, solid]
        table[row sep={\\}]
        {
            \\
            1.0  1.300774920542659  \\
            2.0  1.1370941601767157  \\
            3.0  0.9514402546468679  \\
            4.0  0.6237302263556502  \\
        }
        ;
    \addplot[color={donkerblauw}, only marks, draw opacity={1.0}, line width={0}, solid, mark={*}, mark size={2.0 pt}, mark repeat={1}, mark options={color={rgb,1:red,0.0;green,0.0;blue,0.0}, draw opacity={1.0}, fill={donkerblauw}, fill opacity={1.0}, line width={0.75}, rotate={0}, solid}]
        table[row sep={\\}]
        {
            \\
            1.0  1.300774920542659  \\
            2.0  1.1370941601767157  \\
            3.0  0.9514402546468679  \\
            4.0  0.6237302263556502  \\
        }
        ;
    \addplot[color={donkerblauw}, draw opacity={1.0}, line width={1}, solid]
        table[row sep={\\}]
        {
            \\
            1.0  1.7015830148583757  \\
            2.0  1.537366014525392  \\
            3.0  1.320333055885324  \\
            4.0  1.0574761315138692  \\
            5.0  0.6208679500074303  \\
        }
        ;
    \addplot[color={donkerblauw}, only marks, draw opacity={1.0}, line width={0}, solid, mark={*}, mark size={2.0 pt}, mark repeat={1}, mark options={color={rgb,1:red,0.0;green,0.0;blue,0.0}, draw opacity={1.0}, fill={donkerblauw}, fill opacity={1.0}, line width={0.75}, rotate={0}, solid}]
        table[row sep={\\}]
        {
            \\
            1.0  1.7015830148583757  \\
            2.0  1.537366014525392  \\
            3.0  1.320333055885324  \\
            4.0  1.0574761315138692  \\
            5.0  0.6208679500074303  \\
        }
        ;
    \addplot[color={donkerblauw}, draw opacity={1.0}, line width={1}, solid]
        table[row sep={\\}]
        {
            \\
            1.0  2.0904835031007987  \\
            2.0  1.9700264113263073  \\
            3.0  1.7231417188918703  \\
            4.0  1.4626064532562786  \\
            5.0  1.2220506765769963  \\
            6.0  0.6175223936822627  \\
        }
        ;
    \addplot[color={donkerblauw}, only marks, draw opacity={1.0}, line width={0}, solid, mark={*}, mark size={2.0 pt}, mark repeat={1}, mark options={color={rgb,1:red,0.0;green,0.0;blue,0.0}, draw opacity={1.0}, fill={donkerblauw}, fill opacity={1.0}, line width={0.75}, rotate={0}, solid}]
        table[row sep={\\}]
        {
            \\
            1.0  2.0904835031007987  \\
            2.0  1.9700264113263073  \\
            3.0  1.7231417188918703  \\
            4.0  1.4626064532562786  \\
            5.0  1.2220506765769963  \\
            6.0  0.6175223936822627  \\
        }
        ;
    \node[right, color={donkerblauw}, draw opacity={1.0}, rotate={0.0}, font={\footnotesize}]  at (axis cs:1.5,0.5) {$v_5$};
    \node[right, color={donkerblauw}, draw opacity={1.0}, rotate={0.0}, font={\footnotesize}]  at (axis cs:2.5,0.65) {$v_7$};
    \node[right, color={donkerblauw}, draw opacity={1.0}, rotate={0.0}, font={\footnotesize}]  at (axis cs:3.5,0.82) {$v_9$};
    \node[right, color={donkerblauw}, draw opacity={1.0}, rotate={0.0}, font={\footnotesize}]  at (axis cs:4.5,0.87) {$v_{11}$};
    \node[right, color={donkerblauw}, draw opacity={1.0}, rotate={0.0}, font={\footnotesize}]  at (axis cs:5.5,0.94) {$v_{13}$};
\end{axis}
\end{tikzpicture}%}
\vspace{-4pt}
\caption{The vectors $v_m \in \R^{\floor{\tfrac{m-1}{2}}}$ such that the optimal solution of the dual \eqref{betamdual} of $\beta_m$ is given by $Y = \frac{1}{(m-1)!}v_mv_m\T$. \revise{Note that $v_m$ can be indexed by~$M_i$ ($i=3,\ldots, \lfloor\tfrac{m+1}{2}\rfloor+1$) as in Proposition~\ref{betacombinatorialproposition}. Each plotted function corresponds to the coefficients of one~$v_m$, where a point at position $(M_i, x)$ signifies that the coordinate of $v_m$ corresponding to $M_i$ is $x$.}} \label{figure:betaSolutions}
\end{figure}
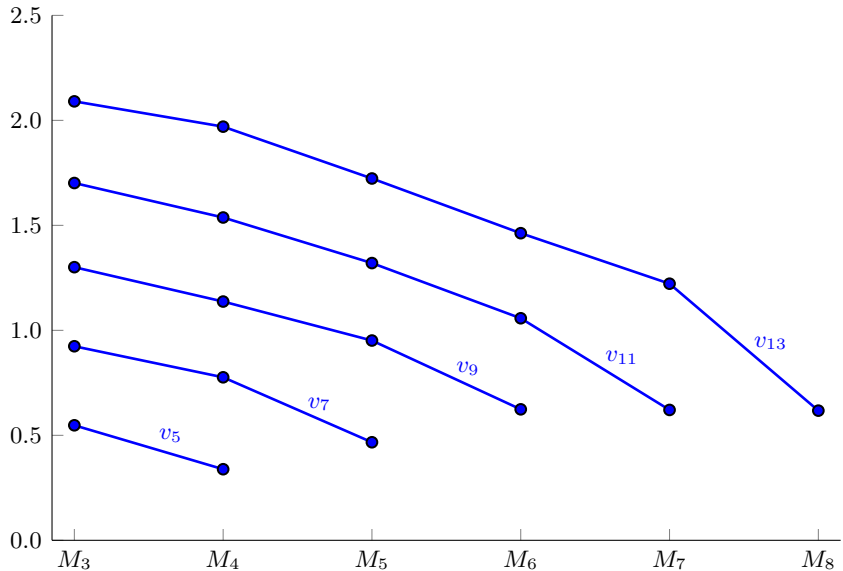
\end{remark}

\subsection{Iterative procedure to obtain the bounds~\texorpdfstring{$\beta_m$}{betam}}
\revise{ To solve~\eqref{betamdual}  on the computer, we follow a cut generation method: First the semidefinite program is solved without the linear constraints. Then:
\begin{itemize} 
 \item All of the constraints are evaluated. (As $m$ grows, this takes up most of the runtime.) %Our SDPs contain exact integer data, and we solve them with sufficiently high precision\footnote{in the sense of TODO} in SDPA-GMP.
 \item We add the most violated constraint as a new constraint to the semidefinite program. When there are ties, we choose the most violated constraint that was evaluated first. 
 \item The semidefinite program is solved again.
 \end{itemize} 
 These steps are repeated, until no constraints are violated anymore.} In theory this procedure could take~$|\Omega_m'|$ iterations. In practice however, the number of iterations is much smaller, and we are able to compute~$\beta_m$ for~$m \leq 13$ up to high precision on a desktop computer --- see Table~\ref{table: boundstable}.\footnote{The julia code used is publicly available via the link: \url{https://github.com/CrossingBounds/CrossingNumber}.}

 \revise{\subsection{Verifying the bounds}\label{sec: verify bounds}}
 \revise{We explain the procedure used to verify our lower bounds. For the bound~$\beta_m$, the starting point is formulation~\eqref{betamdual}. For the bound~$\alpha_m$, one can derive the following analogous formulation. For~$\lambda \vdash m$ and~$\varepsilon \in \{\pm 1\}$, let~$m_{\lambda}^{\varepsilon}$ denote the number of columns of~$U_{\lambda}^{\varepsilon}$ in the representative set for the action of~$S_m \times S_2$  on~$\C^{Z_m}$ we derived in Section 2. Also, for $\omega \in \Omega_m'$, define the matrix~$C_{\omega} := \oplus_{\lambda \vdash m, \, \varepsilon \in \{\pm 1\}} (U_{\lambda}^{\varepsilon})\T K_{\omega} U_{\lambda}^{\varepsilon}$. Then
\begin{align}\label{alphamdual2}
    \alpha_m = \max \Bigg\{ t \, \colon  \, Y \in \hspace{-5pt}\bigoplus_{\substack{\lambda \vdash n\\ \varepsilon \in \{\pm 1 \}}} \R^{m_{\lambda}^{\varepsilon} \times m_{\lambda}^{\varepsilon}}, \,  Y \succeq 0, \,  \forall \omega \in \Omega_m'\, \colon \, \langle Y, C_{\omega}\rangle + |\omega| t \leq |\omega|q_{\omega} \Bigg\}. 
\end{align}
Note that all our SDP's contain integer data after block-diagonalization, so in the SDP-input there is no rounding. However, the high-precision interior-point solution~$(t,Y)$ to~\eqref{betamdual} or~\eqref{alphamdual2} obtained from the solver may exhibit tiny infeasibilities. To obtain a rational feasible solution, we do the following:
\begin{itemize}
    \item Round~$t$ to a rational number~$t'$, and round the eigenvalues~$\lambda_i$ and eigenvectors~$v_i$ of~$Y$ to rationals $\hat \lambda_i$ and rational vectors $\hat v_i$. Construct a new matrix~$Y':=\sum_{i'} \hat\lambda_{i'} \hat v_{i'} \hat v_{i'}\T$ from the nonnegative rounded eigenvalues and the corresponding rounded eigenvectors. Then $Y' \succeq 0$.
    \item Check each of the inequalities (involving only rational numbers) in~\eqref{betamdual} or~\eqref{alphamdual2} using the rational matrix~$Y'$. If the inequality corresponding to~$\omega$ is violated, replace $t'$ by $(|\omega|q_{\omega}-\langle Y', C_{\omega}\rangle)/|\omega|$ so that the inequality is not violated anymore. 
\end{itemize}
In this way, we obtain rational feasible solutions $(t', Y')$ to \eqref{betamdual} or~$\eqref{alphamdual2}$ and thus guaranteed lower bounds on~$\alpha_m$ and $\beta_m$. The obtained lower bounds coincide with the approximations of~$\alpha_m$ and~$\beta_m$ computed by the solver for all decimals given in Table~\ref{table: boundstable}. (At least 40 decimals are correct for all computed bounds except $\alpha_{10}$ using SDPA-GMP~\cite{nakata}, and at least 13 decimals are correct for~$\alpha_{10}$ using SDPA-DD.)}

%-make a function roundRationalPSD, the solution matrix Y.  from (19) 
%If it is negative, round up to 0. Otherwise: compute eigenvectors, values: round to rationals. Delete negative eigenvalues and make new matrix Y' from rounded eigenvectors. Guaranteed to be PSD and rational.
%-compute all eigenvalues and vectors.
%-go through constraints. And reformulate so that t <= omega q - <Y,Aomega)/omega.  If needed: shift t so that inequality is satisfied. 

\section*{Acknowledgements} The authors thank Sander Gribling, Etienne de Klerk, Monique Laurent, Bart Litjens and Lex Schrijver for useful discussions. \revise{The authors also thank the anonymous referees and the editor for their careful reading and valuable comments to improve the content and presentation of the paper, as well as the proofs of Propositions~\ref{representativeprop} and~\ref{s2proposition}.

Most of this research was carried out while D.\ Brosch was with Tilburg University, Tilburg and S.\ Polak was with Centrum Wiskunde \& Informatica, Amsterdam.}

%\selectlanguage{english} 

\end{document}